\documentclass{amsart}

\usepackage[all]{xypic}
\usepackage[centertags]{amsmath}
\usepackage{amsfonts}
\usepackage{amscd}
\usepackage{amssymb}
\usepackage{amsthm}
\usepackage{newlfont}
\usepackage{amsxtra}
\usepackage{verbatim}

 \newtheorem{thm}{Theorem}[section]
 \newtheorem{cor}[thm]{Corollary}
 \newtheorem{conj}[thm]{Conjecture}
 \newtheorem{lem}[thm]{Lemma}
 \newtheorem{prop}[thm]{Proposition}
 \theoremstyle{definition}
 \newtheorem{defn}[thm]{Definition}
 \theoremstyle{remark}
 \newtheorem{rem}[thm]{Remark}
 \theoremstyle{remark}
 
 \theoremstyle{definition}
 \newtheorem{notn}[thm]{Notation}
 \numberwithin{equation}{section}

 \newcommand{\an}{\mathrm{an}}

 \newcommand{\Hom}{\mathrm{Hom}}
 
 \newcommand{\Ext}{\mathrm{Ext}}
 \newcommand{\Spec}{\mathrm{Spec}}
 \newcommand{\Frob}{\mathrm{Frob}}

 \newcommand{\Sp}{\mathrm{Sp}}
 \newcommand{\Aut}{\mathrm{Aut}}
 \newcommand{\End}{\mathrm{End}}
 \newcommand{\Pic}{\mathrm{Pic}}
 
 \newcommand{\ord}{\mathrm{ord}}

 \newcommand{\Gal}{\mathrm{Gal}}
 \newcommand{\GL}{\mathrm{GL}}
 \newcommand{\PGL}{\mathrm{PGL}}
 
 \newcommand{\sep}{\mathrm{sep}}
 \newcommand{\coker}{\mathrm{coker}}

 \newcommand{\Ram}{\mathrm{Ram}}

 \newcommand{\Stab}{\mathrm{Stab}}

 \newcommand{\new}{\mathrm{new}}

 \newcommand{\Id}{\mathrm{Id}}
 
 \newcommand{\et}{\mathrm{et}}

 \newcommand{\Nr}{\mathrm{Nr}}
 \newcommand{\Odd}{\mathrm{Odd}}
 
 \newcommand{\fa}{\mathfrak a}

 \newcommand{\fp}{\mathfrak p}

 \newcommand{\fn}{\mathfrak n}
 \newcommand{\fm}{\mathfrak m}

 \newcommand{\fD}{\mathfrak D}

 \newcommand{\cO}{\mathcal{O}}

 \newcommand{\cK}{\mathcal{K}}
 
 \newcommand{\cC}{\mathcal{C}}
 \newcommand{\cA}{\mathcal{A}}
 
 \newcommand{\cG}{\mathcal{G}}
 
 \newcommand{\cS}{\mathcal{S}}
 
 \renewcommand{\cR}{\mathcal{R}}
 \renewcommand{\cD}{\mathcal{D}}
 \newcommand{\cE}{\mathcal{E}}

 \newcommand{\cI}{\mathcal{I}}
 
 \renewcommand{\cH}{\mathcal{H}}
 \newcommand{\cT}{\mathcal{T}}
 \newcommand{\gm}[1]{\mathbb{G}_{m,#1}}

 \newcommand{\C}{\mathbb{C}}
 \newcommand{\E}{\mathbb{E}}
 \newcommand{\F}{\mathbb{F}}
 \newcommand{\M}{\mathbb{M}}
 \newcommand{\Q}{\mathbb{Q}}
 \newcommand{\Z}{\mathbb{Z}}
 \newcommand{\A}{\mathbb{A}}

 \renewcommand{\P}{\mathbb{P}}

 \newcommand{\Ell}{\mathcal{E}\ell\ell}

 \newcommand{\To}{\longrightarrow}
 
 \newcommand{\bs}{\setminus}

 \newcommand{\tE}{\widetilde{E}}
 
 \newcommand{\tX}{\widetilde{X}}

 \newcommand{\Fi}{F_\infty}

 \newcommand{\G}{\Gamma}
 
 \newcommand{\La}{\Lambda}

 \newcommand{\twist}[1]{{^\tau}\!#1}

\begin{document}

\title[Jacquet-Langlands isogeny over function fields]
{On Jacquet-Langlands isogeny over function fields}

\author{Mihran Papikian}

\address{Department of Mathematics, Pennsylvania State University, University Park, PA 16802}

\email{papikian@math.psu.edu}

\thanks{The author was supported in part by NSF grant DMS-0801208.}




\begin{abstract}
We propose a conjectural explicit isogeny from the Jacobians of
hyperelliptic Drinfeld modular curves to the Jacobians of
hyperelliptic modular curves of $\cD$-elliptic sheaves. The kernel
of the isogeny is a subgroup of the cuspidal divisor group
constructed by examining the canonical maps from the cuspidal
divisor group into the component groups.
\end{abstract}


\maketitle


\section{Introduction} Let $N$ be a square-free integer, divisible by an even
number of primes. It is well-known that the new part of the modular
Jacobian $J_0(N)$ is isogenous to the Jacobian of a Shimura curve;
see \cite{RibetIsogeny}. The existence of this isogeny can be
interpreted as a geometric incarnation of the global
Jacquet-Langlands correspondence over $\Q$ between the cusp forms on
$\GL(2)$ and the multiplicative group of a quaternion algebra
\cite{JaLa}. Jacquet-Langlands isogeny has important arithmetic
applications, for example, to level lowering \cite{Ribet90}. In this
paper we are interested in the function field analogue of the
Jacquet-Langlands isogeny.

\vspace{0.1in}

Let $\F_q$ be the finite field with $q$ elements, and let
$F=\F_q(T)$ be the field of rational functions on $\P^1_{\F_q}$. The
set of places of $F$ will be denoted by $|F|$. Let $A:=\F_q[T]$.
This is the subring of $F$ consisting of functions which are regular
away from the place generated by $1/T$ in $\F_q[1/T]$. The place
generated by $1/T$ will be denoted by $\infty$ and called the
\textit{place at infinity}; it will play a role similar to the
archimedean place for $\Q$. The places in $|F|-\infty$ are the
\textit{finite places}.

Let $v\in |F|$. We denote by $F_v$, $\cO_v$ and $\F_v$ the
completion of $F$ at $v$, the ring of integers in $F_v$, and the
residue field of $F_v$, respectively. We assume that the valuation
$\ord_v: F_v\to \Z$ is normalized by $\ord_v(\pi_v)=1$, where
$\pi_v$ is a uniformizer of $\cO_v$. The \textit{degree of $v$} is
$\deg(v)=[\F_v:\F_q]$. Let $q_v:=q^{\deg(v)}=\# \F_v$. If $v$ is a
finite place, then with an abuse of notation we denote the prime
ideal of $A$ corresponding to $v$ by the same letter.

Given a field $K$, we denote by $\bar{K}$ an algebraic closure of
$K$.

Let $R\subset |F|-\infty$ be a nonempty finite set of places of even
cardinality. Let $D$ be the quaternion algebra over $F$ ramified
exactly at the places in $R$. Let $X^R_F$ be the modular curve of
$\cD$-elliptic sheaves; see $\S$\ref{ssDES}. This curve is the
function field analogue of a Shimura curve parametrizing abelian
surfaces with multiplication by a maximal order in an indefinite
division quaternion algebra over $\Q$. Denote the Jacobian of
$X^R_F$ by $J^R$. The role of classical modular curves in this
context is played by Drinfeld modular curves. With an abuse of
notation, let $R$ also denote the square-free ideal of $A$ whose
support consists of the places in $R$. Let $X_0(R)_F$ be the
Drinfeld modular curve defined in $\S$\ref{ssDMC}. Let $J_0(R)$ be
the Jacobian of $X_0(R)_F$. The same strategy as over $\Q$ shows
that $J^R$ is isogenous to the new part of $J_0(R)$ (see Theorem
\ref{thmExistence} and Remark \ref{rem-new}). The proof relies on
Tate's conjecture, so it provides no information about the isogenies
$J^R\to J_0(R)^\new$ beyond their existence. In this paper we
carefully examine the simplest non-trivial case, namely $R=\{x, y\}$
with $\deg(x)=1$ and $\deg(y)=2$. (When $R=\{x, y\}$ and
$\deg(x)=\deg(y)=1$, both $X^R_F$ and $X_0(R)_F$ have genus $0$.)

\begin{notn} Unless indicated otherwise, throughout the paper $x$ and $y$ will be
two fixed finite places of degree $1$ and $2$, respectively. When
$R=\{x,y\}$, we write $X^{xy}_F$ for $X^R_F$, $J^{xy}$ for $J^R$,
$X_0(xy)_F$ for $X_0(R)_F$, and $J_0(xy)$ for $J_0(R)$.
\end{notn}

The genus of $X^{xy}_F$ is $q$, which is also the genus of
$X_0(xy)_F$. Hence $J_0(xy)$ and $J^{xy}$ are $q$-dimensional
Jacobian varieties, which are isogenous over $F$. We would like to
construct an explicit isogeny $J_0(xy)\to J^{xy}$. A natural place
to look for the kernel of an isogeny defined over $F$ is in the
cuspidal divisor group $\cC$ of $J_0(xy)$. To see which subgroup of
$\cC$ could be the kernel, one needs to compute, besides $\cC$
itself, the component groups of $J_0(xy)$ and $J^{xy}$, and the
canonical specialization maps of $\cC$ into the component groups of
$J_0(xy)$. These calculations constitute the bulk of the paper.
Based on these calculations, in $\S$\ref{SecJLisog} we propose a
conjectural explicit isogeny $J_0(xy)\to J^{xy}$, and prove that the
conjecture is true for $q=2$. We note that $X^{xy}_F$ is
hyperelliptic, and in fact for odd $q$ these are the only $X^R_F$
which are hyperelliptic \cite{PapHE}. The curve $X_0(xy)_F$ is also
hyperelliptic, and for levels which decompose into a product of two
prime factors these are the only hyperelliptic Drinfeld modular
curves \cite{SchweizerHE}. Hence this paper can also be considered
as a study of hyperelliptic modular Jacobians over $F$ which
interrelates \cite{PapHE} and \cite{SchweizerHE}.

The approach to explicating the Jacquet-Langlands isogeny through
the study of component groups and cuspidal divisor groups was
initiated in the classical context by Ogg. In \cite{Ogg}, Ogg
proposed in several cases conjectural explicit isogenies between the
modular Jacobians and the Jacobians of Shimura curves (as far as I
know, these conjectures are still mostly open, but see \cite{GoRo}
and \cite{Helm} for some advances).

\vspace{0.1in}

We summarize the main results of the paper.
\begin{itemize}
\item The cuspidal divisor group $\cC\subset J_0(xy)(F)$ is
isomorphic to
$$
\cC\cong \Z/(q+1)\Z \oplus \Z/(q^2+1)\Z.
$$
\item The component groups of $J_0(xy)$ and $J^{xy}$ at
$x$, $y$, and $\infty$ are listed in Table \ref{table0}. 
($J_0(xy)$ and $J^{xy}$ have good reduction away from $x$, $y$ and
$\infty$, so the component groups are trivial away from these three
places.)
\begin{table}
\begin{tabular}{|c|c|c|c|}
\hline
 & $x$ & $y$ & $\infty$\\
\hline
$J_0(xy)$ & $\Z/(q^2+1)(q+1)\Z$ & $\Z/(q+1)\Z$ & $\Z/(q^2+1)(q+1)\Z$ \\
\hline
$J^{xy}$ & $\Z/(q+1)\Z$ & $\Z/(q^2+1)(q+1)\Z$ & $\Z/(q+1)\Z$\\
\hline
\end{tabular}\caption{Component groups}\label{table0}
\end{table}
\item If we denote the component group of $J_0(xy)$ at $\ast$ by
$\Phi_\ast$, and the canonical map $\cC\to \Phi_\ast$ by
$\phi_\ast$, then there are exact sequences
$$
0\to \Z/(q+1)\Z\to \cC\xrightarrow{\phi_x} \Phi_x\to \Z/(q+1)\Z\to
0,
$$
$$
0\to \Z/(q^2+1)\Z\to \cC\xrightarrow{\phi_y} \Phi_y\to 0,
$$
$$
\phi_\infty: \cC\overset{\sim}{\To} \Phi_\infty\quad \text{if $q$ is
even},
$$
$$
0\to \Z/2\Z\to \cC\xrightarrow{\phi_\infty}\Phi_\infty\to \Z/2\Z\to
0\quad \text{if $q$ is odd}.
$$
\item The kernel $\cC_0\cong \Z/(q^2+1)\Z$ of $\phi_y$ maps
injectively into $\Phi_x$ and $\Phi_\infty$
\end{itemize}
Conjecture \ref{myConj} then states that there is an isogeny
$J_0(xy)\to J^{xy}$ whose kernel is $\cC_0$. As an evidence for the
conjecture, we prove that the quotient abelian variety
$J_0(xy)/\cC_0$ has component groups of the same order as $J^{xy}$.
This is a consequence of a general result (Theorem \ref{propGroth}),
which describes how the component groups of abelian varieties with
toric reduction change under isogenies. Finally, we prove Conjecture
\ref{myConj} for $q=2$ (Theorem \ref{propq=2}); the proof relies on
the fact that $J_0(xy)$ in this case is isogenous to a product of
two elliptic curves. Two other interesting consequences of our
results are the following. First, we deduce the genus formula for
$X^R_F$ proven in \cite{PapGenus} by a different argument (Corollary
\ref{corGenus}). Second, assuming $q$ is even and Conjecture
\ref{myConj} is true, we are able to tell how the optimal elliptic
curve with conductor $xy\infty$ changes in a given $F$-isogeny class
when we change the modular parametrization from $X_0(xy)_F$ to
$X^{xy}_F$ (Proposition \ref{propOpt}).


\section{Preliminaries}\label{SecPrelim}


\subsection{Drinfeld modular curves}\label{ssDMC} Let $K$ be an $A$-field, i.e.,
$K$ is a field equipped with a homomorphism $\gamma: A\to K$. In
particular, $K$ contains $\F_q$ as a subfield. The
\textit{$A$-characteristic} of $K$ is the ideal $\ker(\gamma)\lhd
A$. Let $K\{\tau\}$ be the twisted polynomial ring with commutation
rule $\tau s=s^q\tau$, $s\in K$. A \textit{rank-$2$ Drinfeld
$A$-module over $K$} is a ring homomorphism $\phi: A\to K\{\tau\}$,
$a\mapsto \phi_a$ such that $\deg_\tau \phi_a=-2\ord_\infty(a)$ and
the constant term of $\phi_a$ is $\gamma(a)$. A
\textit{homomorphism} of two Drinfeld modules $u:\phi\to \psi$ is
$u\in K\{\tau\}$ such that $\phi_a u=u\psi_a$ for all $a$ in $A$;
$u$ is an \textit{isomorphism} if $u\in K^\times$. Note that $\phi$
is uniquely determined by the image of $T$:
$$
\phi_T=\gamma(T)+g\tau+\Delta\tau^2,
$$
where $g\in K$ and $\Delta\in K^\times$. The \textit{$j$-invariant}
of $\phi$ is $j(\phi)=g^{q+1}/\Delta$. It is easy to check that if
$K$ is algebraically closed, then $\phi\cong \psi$ if and only if
$j(\phi)=j(\psi)$.

Treating $\tau$ as the automorphism of $K$ given by $k\mapsto k^q$,
the field $K$ acquires a new $A$-module structure via $\phi$. Let
$\fa\lhd A$ be an ideal. Since $A$ is a principal ideal domain, we
can choose a generator $a\in A$ of $\fa$. The $A$-module
$\phi[\fa]=\ker\phi_a(\bar{K})$ does not depend on the choice of $a$
and is called the \textit{$\fa$-torsion} of $\phi$. If $\fa$ is
coprime to the $A$-characteristic of $K$, then $\phi[\fa]\cong
(A/\fa)^2$. On the other hand, if $\fp=\ker(\gamma)\neq 0$, then
$\phi[\fp]\cong (A/\fp)$ or $0$; when $\phi[\fp]= 0$, $\phi$ is
called \textit{supersingular}.

\begin{lem}\label{lemPrelim1}
Up to isomorphism, there is a unique supersingular rank-$2$ Drinfeld $A$-module over
$\overline{\F}_x$: it is the Drinfeld module with $j$-invariant
equal to $0$. Up to isomorphism, there is a unique supersingular rank-$2$ Drinfeld
$A$-module over $\overline{\F}_y$, and its $j$-invariant is
non-zero.
\end{lem}
\begin{proof}
This follows from \cite[(5.9)]{GekelerADM} since $\deg(x)=1$ and
$\deg(y)=2$.
\end{proof}

Let $\End(\phi)$ denote the centralizer of $\phi(A)$ in
$\bar{K}\{\tau\}$, i.e., the ring of all homomorphisms $\phi\to
\phi$ over $\bar{K}$. The automorphism group $\Aut(\phi)$ is the
group of units $\End(\phi)^\times$.

\begin{lem}\label{lemPrelim2} If $j(\phi)\neq 0$, then
$\Aut(\phi)\cong \F_q^\times$. If $j(\phi)= 0$, then
$\Aut(\phi)\cong\F_{q^2}^\times$.
\end{lem}
\begin{proof}
If $u\in \bar{K}^\times$ commutes with $\phi_T=\gamma(T)+g\tau
+\Delta \tau^2$, then $u^{q^2-1}=1$ and $u^{q-1}=1$ if $g\neq 0$.
This implies that $u\in \F_q^\times$ if $j(\phi)\neq 0$, and $u\in
\F_{q^2}^\times$ if $j(\phi)=0$. On the other hand, we clearly have
the inclusions $\F_q^\times\subset \Aut(\phi)$ and, if $j(\phi)=0$,
$\F_{q^2}^\times\subset \Aut(\phi)$. This finishes the proof.
\end{proof}

\begin{lem}\label{lemPrelim3} Let $\fp\lhd A$ be a prime ideal and $\F_\fp:=A/\fp$.
Let $\phi$ be a rank-$2$ Drinfeld $A$-module over
$\overline{\F}_\fp$. Let $\fn\lhd A$ be an ideal coprime to $\fp$.
Let $C_\fn$ be an $A$-submodule of $\phi[\fn]$ isomorphic to
$A/\fn$. Denote by $\Aut(\phi, C_\fn)$ the subgroup of automorphisms
of $\phi$ which map $C_\fn$ to itself. Then $\Aut(\phi, C_\fn)\cong
\F_q^\times$ or $\F_{q^2}^\times$. The second case is possible only
if $j(\phi)=0$.
\end{lem}
\begin{proof}
The action of $\F_q^\times$ obviously stabilizes $C_\fn$, hence,
using Lemma \ref{lemPrelim2}, it is enough to show that if
$\Aut(\phi, C_\fn)\neq \F_q^\times$, then $\Aut(\phi, C_\fn)\cong
\F_{q^2}^\times$. Let $u\in \Aut(\phi, C_\fn)$ be an element which
is not in $\F_q$. Then $\Aut(\phi)=\F_q[u]^\times\cong
\F_{q^2}^\times$, where $\F_q[u]$ is considered as a finite subring
of $\End(\phi)$. It remains to show that $\alpha+u\beta$ stabilizes
$C_\fn$ for any $\alpha,\beta\in \F_q$ not both equal to zero. But
this is obvious since $\alpha$ and $u\beta$ stabilize $C_\fn$ and
$C_\fn\cong A/\fn$ is cyclic.
\end{proof}

One can generalize the notion of Drinfeld modules over an $A$-field
to the notion of Drinfeld modules over an arbitrary $A$-scheme $S$
\cite{Drinfeld}. The functor which associates to an $A$-scheme $S$
the set of isomorphism classes of pairs $(\phi,C_\fn)$, where $\phi$
is a Drinfeld $A$-module of rank $2$ over $S$ and $C_\fn\cong A/\fn$
is an $A$-submodule of $\phi[\fn]$, possesses a coarse moduli scheme
$Y_0(\fn)$ that is affine, flat and of finite type over $A$ of pure
relative dimension $1$. There is a canonical compactification
$X_0(\fn)$ of $Y_0(\fn)$ over $\Spec(A)$; see \cite[$\S$9]{Drinfeld}
or \cite{vdPT}. The finitely many points
$X_0(\fn)(\bar{F})-Y_0(\fn)(\bar{F})$ are called the \textit{cusps}
of $X_0(\fn)_F$.

Denote by $\C_\infty$ the completion of an algebraic closure of
$\Fi$. Let $\Omega=\C_\infty-\Fi$ be the \textit{Drinfeld upper
half-plane}; $\Omega$ has a natural structure of a smooth connected
rigid-analytic space over $\Fi$. Denote by $\G_0(\fn)$ the
\textit{Hecke congruence subgroup} of level $\fn$:
$$
\G_0(\fn)=\left\{\begin{pmatrix} a & b \\ c & d
\end{pmatrix}\in \GL_2(A)\ |\ c\in \fn\right\}.
$$
The group $\G_0(\fn)$ naturally acts on $\Omega$ via linear
fractional transformations, and the action is \textit{discrete} in
the sense of \cite[p. 582]{Drinfeld}. Hence we may construct the
quotient $\G_0(\fn)\bs \Omega$ as a $1$-dimensional connected smooth
analytic space over $\Fi$.

The following theorem can be deduced from the results in
\cite{Drinfeld}:

\begin{thm}\label{Drinf} $X_0(\fn)$ is a proper flat scheme of pure
relative dimension $1$ over $\Spec(A)$, which is smooth away from
the support of $\fn$. There is an isomorphism of rigid-analytic
spaces $\G_0(\fn)\bs \Omega\cong Y_0(\fn)_{\Fi}^\an$.
\end{thm}

There is a genus formula for $X_0(\fn)_F$ which depends on the prime
decomposition of $\fn$; see \cite[Thm. 2.17]{GN}. By this formula,
the genera of $X_0(x)_F$, $X_0(y)_F$ and $X_0(xy)_F$ are $0$, $0$
and $q$, respectively.


\subsection{Modular curves of $\cD$-elliptic sheaves}\label{ssDES}
Let $D$ be a quaternion algebra over $F$. Let $R\subset |F|$ be the
set of places which ramify in $D$, i.e., $D\otimes F_v$ is a
division algebra for $v\in R$. It is known that $R$ is finite of
even cardinality, and, up to isomorphism, this set uniquely determines $D$; see
\cite{Vigneras}. Assume $R\neq \emptyset$ and $\infty\not \in R$. In
particular, $D$ is a division algebra. Let $C:=\P^1_{\F_q}$. Fix a
locally free sheaf $\cD$ of $\cO_C$-algebras with stalk at the
generic point equal to $D$ and such that
$\cD_v:=\cD\otimes_{\cO_C}\cO_v$ is a maximal order in
$D_v:=D\otimes_F F_v$.

Let $S$ be an $\F_q$-scheme. Denote by $\Frob_S$ its Frobenius
endomorphism, which is the identity on the points and the $q$th
power map on the functions. Denote by $C\times S$ the fiberd product
$C\times_{\Spec(\F_q)}S$. Let $z:S\to C$ be a morphism of
$\F_q$-schemes. A \textit{$\cD$-elliptic sheaf over $S$}, with pole
$\infty$ and zero $z$, is a sequence $\E=(\cE_i,j_i,t_i)_{i\in \Z}$,
where each $\cE_i$ is a locally free sheaf of $\cO_{C\times
S}$-modules of rank $4$ equipped with a right action of $\cD$
compatible with the $\cO_C$-action, and where
\begin{align*}
j_i &:\cE_i\to \cE_{i+1}\\
t_i &:\twist{\cE}_{i}:=(\Id_C\times \Frob_{S})^\ast \cE_i\to
\cE_{i+1}
\end{align*}
are injective $\cO_{C\times S}$-linear homomorphisms compatible with
the $\cD$-action. The maps $j_i$ and $t_i$ are sheaf modifications
at $\infty$ and $z$, respectively, which satisfy certain conditions,
and it is assumed that for each closed point $w$ of $S$, the
Euler-Poincar\'e characteristic $\chi(\cE_0|_{C\times w})$ is in the
interval $[0,2)$; we refer to \cite[$\S$2]{LRS} and
\cite[$\S$1]{Hausberger} for the precise definition. Moreover, to
obtain moduli schemes with good properties at the closed points $w$
of $S$ such that $z(w)\in R$ one imposes an extra condition on $\E$
to be ``special'' \cite[p. 1305]{Hausberger}. Note that, unlike the
original definition in \cite{LRS}, $\infty$ is allowed to be in the
image of $S$; here we refer to \cite[$\S$4.4]{BS} for the details.
Denote by $\Ell^\cD(S)$ the set of isomorphism classes of
$\cD$-elliptic sheaves over $S$. The following theorem can be
deduced from some of the main results in \cite{LRS} and
\cite{Hausberger}:

\begin{thm}\label{thmMC}
The functor $S\mapsto \Ell^{\cD}(S)$ has a coarse moduli scheme
$X^R$, which is proper and flat of pure relative dimension $1$ over
$C$ and is smooth over $C-R-\infty$.
\end{thm}

\begin{rem}
Theorems \ref{Drinf} and \ref{thmMC} imply that $J_0(R)$ and $J^R$
have good reduction at any place $v\in |F|-R-\infty$; cf. \cite[Ch.
9]{NM}.
\end{rem}


\section{Cuspidal divisor group}\label{SecCDG}

For a field $K$, we represent the elements of $\P^1(K)$ as column
vectors $\begin{pmatrix} u
\\ v\end{pmatrix}$ where $u, v\in K$ are not both zero
and $\begin{pmatrix} u \\
v\end{pmatrix}$ is identified with $\begin{pmatrix} \alpha u \\
\alpha v\end{pmatrix}$ if $\alpha \in K^\times$. We assume that
$\GL_2(K)$ acts on $\P^1(K)$ on the left by
$$
\begin{pmatrix} a & b \\ c & d\end{pmatrix}\begin{pmatrix} u
\\ v\end{pmatrix}=\begin{pmatrix} au+bv
\\ cu+dv\end{pmatrix}.
$$

Let $\fn\lhd A$ be an ideal. The cusps of $X_0(\fn)_F$ are in
natural bijection with the orbits of $\G_0(\fn)$ acting from the
left on $\P^1(F)$.

\begin{lem} If $\fn$ is square-free, then there are $2^s$ cusps on
$X_0(\fn)_F$, where $s$ is the number of prime divisors of $\fn$.
All the cusps are $F$-rational.
\end{lem}
\begin{proof}
See Proposition 3.3 and Corollary 3.4 in \cite{GekelerUber}.
\end{proof}

For every $\fm|\fn$ with $(\fm, \fn/\fm)=1$ there is an
\textit{Atkin-Lehner involution} $W_\fm$ on $X_0(\fn)_F$, cf.
\cite{SchweizerHE}. Its action is given by multiplication from the
left with any matrix $\begin{pmatrix} ma & b\\ n & m \end{pmatrix}$
whose determinant generates $\fm$, and where $a,b, m, n\in A$,
$(n)=\fn$, $(m)=\fm$.

From now on assume $\fn=xy$. Recall that we denote by $x$ and $y$
the prime ideals of $A$ corresponding to the places $x$ and $y$,
respectively. With an abuse of notation, we will denote by $x$ also
the monic irreducible polynomial in $A$ generating the ideal $x$,
and similarly for $y$. It should be clear from the context in which
capacity $x$ and $y$ are being used. With this notation, $X_0(xy)_F$
has $4$ cusps, which can be represented by
$$
[\infty]:=\begin{pmatrix} 1 \\ 0\end{pmatrix}, \quad
[0]:=\begin{pmatrix} 0
\\ 1\end{pmatrix},
\quad [x]:=\begin{pmatrix} 1 \\ x\end{pmatrix}, \quad [y]:=\begin{pmatrix} 1 \\
y\end{pmatrix},
$$
cf. \cite[p. 333]{SchweizerHE} and \cite[p. 196]{GekelerDisc}.

There are $3$ non-trivial Atkin-Lehner involutions $W_x, W_y$,
$W_{xy}$ which generate a group isomorphic to $(\Z/2\Z)^2$: these
involutions commute with each other and satisfy
$$
W_xW_y=W_{xy}, \quad W_x^2=W_y^2=W_{xy}^2=1.
$$
By \cite[Prop. 9]{SchweizerHE}, none of these involutions fixes a
cusp. In fact, a simple direct calculation shows that
\begin{align}\label{eq-AL}
W_{xy}([\infty])&=[0], \quad W_{xy}([x])=[y];\\
\nonumber W_x ([\infty])&= [y], \quad W_x([0])=[x];\\
\nonumber W_y ([\infty])&= [x], \quad W_y([0])=[y].
\end{align}

Let $\Delta(z)$, $z\in \Omega$, denote the Drinfeld discriminant
function; see \cite{GekelerUber} or \cite{GekelerDisc} for the
definition. This is a holomorphic and nowhere vanishing function on
$\Omega$. In fact, $\Delta(z)$ is a type-$0$ and weight-$(q^2-1)$
cusp form for $\GL_2(A)$. Its order of vanishing at the cusps of
$X_0(\fn)_F$ can be calculated using \cite{GekelerDisc}. When
$\fn=xy$, \cite[(3.10)]{GekelerDisc} implies
\begin{equation}\label{eq-ordDelta}
\ord_{[\infty]}\Delta=1,\quad \ord_{[0]}\Delta=q_xq_y, \quad
\ord_{[x]}\Delta=q_y, \quad \ord_{[y]}\Delta=q_x.
\end{equation}
The functions
$$
\Delta_{x}(z):=\Delta(xz),\quad \Delta_{y}(z):=\Delta(yz), \quad
\Delta_{xy}(z):=\Delta(xyz)
$$
are type-$0$ and weight-$(q^2-1)$ cusp forms for $\G_0(xy)$. Hence
the fractions $\Delta/\Delta_x$, $\Delta/\Delta_y$,
$\Delta/\Delta_{xy}$ define rational functions on $X_0(xy)_{\C_\infty}$. We
compute the divisors of these functions.

The matrix $W_{xy}=\begin{pmatrix} 0 & 1\\ xy  & 0\end{pmatrix} $
normalizes $\G_0(xy)$ and interchanges $\Delta(z)$ and
$\Delta_{xy}(z)$. Thus by (\ref{eq-AL}) and (\ref{eq-ordDelta})
$$
\ord_{[\infty]}\Delta_{xy}=q_xq_y,\quad \ord_{[0]}\Delta_{xy}=1,
\quad \ord_{[x]}\Delta_{xy}=q_x, \quad \ord_{[y]}\Delta_{xy}=q_y.
$$
A similar argument involving the actions of $W_x$ and $W_y$ gives
$$
\ord_{[\infty]}\Delta_{x}=q_x,\quad \ord_{[0]}\Delta_{x}=q_y, \quad
\ord_{[x]}\Delta_{x}=q_xq_y, \quad \ord_{[y]}\Delta_{x}=1;
$$
$$
\ord_{[\infty]}\Delta_{y}=q_y,\quad \ord_{[0]}\Delta_{y}=q_x, \quad
\ord_{[x]}\Delta_{y}=1, \quad \ord_{[y]}\Delta_{y}=q_xq_y.
$$
From these calculations we obtain
\begin{align*}
\mathrm{div}(\Delta/\Delta_{xy})&=
(1-q_xq_y)[\infty]+(q_xq_y-1)[0]+(q_y-q_x)[x]+(q_x-q_y)[y]\\
&= (q^3-1)([0]-[\infty])+(q^2-q)([x]-[y]),
\end{align*}
and similarly,
$$
\mathrm{div}(\Delta/\Delta_{x}) =
(q-1)([y]-[\infty])+(q^3-q^2)([0]-[x]),
$$
$$
\mathrm{div}(\Delta/\Delta_{y})=
(q^2-1)([x]-[\infty])+(q^3-q)([0]-[y]).
$$

Next, by \cite[p. 200]{GekelerDisc}, the largest positive integer
$k$ such that $\Delta/\Delta_{xy}$ has a $k$th root in the field of
modular functions for $\G_0(xy)$ is $(q-1)^2/(q-1)=(q-1)$. We can
apply the same argument to $\Delta/\Delta_{x}$ as a modular function 
for $\G_0(x)$ to deduce that $\Delta/\Delta_{x}$ has
$(q-1)^2/(q-1)$th root. Similarly, $\Delta/\Delta_{y}$ has
$(q-1)(q^2-1)/(q-1)$th root. Therefore, the following relations
hold in $\Pic^0(X_0(xy)_F)$:
\begin{align}\label{eq-rel1}
(q^2+q+1)([0]-[\infty])+q([x]-[y])&=0\\
\nonumber ([y]-[\infty])+q^2([0]-[x])&=0\\
\nonumber ([x]-[\infty])+q([0]-[y])&=0.
\end{align}

There is one more relation between the cuspidal divisors which comes
from the fact that $X_0(xy)_F$ is hyperelliptic. By a theorem of
Schweizer \cite[Thm. 20]{SchweizerHE}, $X_0(xy)_F$ is hyperelliptic,
and $W_{xy}$ is the hyperelliptic involution. Consider the
degree-$2$ covering
$$\pi: X_0(xy)_F\to X_0(xy)_F/W_{xy}\cong \P^1_{F}.$$ Denote $P:=\pi([\infty]), Q:=\pi([x])$.
Since $W_{xy}([\infty])\neq [x]$, $P\neq Q$. There is a function $f$
on $\P^1_F$ with divisor $P-Q$. Now
$$
\mathrm{div}(\pi^\ast f)=\pi^\ast(\mathrm{div}(f))=\pi^\ast(P-Q)
$$
$$
=([\infty]+W_{xy}([\infty]))-([x]+W_{xy}([x]))=[\infty]+[0]-[x]-[y].
$$
This gives the relation in $\Pic^0(X_0(xy)_F)$
\begin{equation}\label{eq-rel2}
[\infty]+[0]-[x]-[y]=0.
\end{equation}

Fixing $[\infty]\in X_0(xy)(F)$ as an $F$-rational point, we have
the Abel-Jacobi map $X_0(xy)_F\to J_0(xy)$ which sends a point $P\in
X_0(xy)_F$ to the linear equivalence class of the degree-$0$ divisor
$P-[\infty]$.

\begin{defn}
Let $c_0, c_x, c_y\in J_0(xy)(F)$ be the classes of $[0]-[\infty]$,
$[x]-[\infty]$, and $[y]-[\infty]$, respectively. These give
$F$-rational points on the Jacobian since the cusps are
$F$-rational. The \textit{cuspidal divisor group} is the subgroup
$\cC\subset J_0(xy)$ generated by $c_0, c_x$, and $c_y$.
\end{defn}

From (\ref{eq-rel1}) and (\ref{eq-rel2}) we obtain the following
relations:
\begin{align*}
(q^2+q+1)c_0+qc_x-qc_y&=0\\
q^2c_0 - q^2c_x+c_y&=0\\
qc_0+c_x-qc_y&=0\\
c_0-c_x-c_y&=0.
\end{align*}

\begin{lem}\label{lemFApprx}
The cuspidal divisor group $\cC$ is 
generated by $c_x$ and $c_y$, which have orders dividing 
$q+1$ and $q^2+1$, 
respectively.
\end{lem}
\begin{proof}
Substituting $c_0=c_x+c_y$ into the first three equations above, we see   
that $\cC$ is
generated by $c_x$ and $c_y$ subject to relations:
\begin{align*}
(q+1)c_x&=0\\
(q^2+1)c_y&=0.
\end{align*}
\end{proof}

The following simple lemma, which will be used later on, shows that
the factors $(q^2+1)$ and $(q+1)$ appearing in Lemma \ref{lemFApprx}
are almost coprime.

\begin{lem}\label{lem_gcd}
Let $n$ be a positive integer. Then
$$
\mathrm{gcd}(n^2+1, n+1)=\left\{
  \begin{array}{ll}
    1, & \hbox{if $n$ is even;} \\
    2, & \hbox{if $n$ is odd.}
  \end{array}
\right.
$$
\end{lem}
\begin{proof}
Let $d=\mathrm{gcd}(n^2+1, n+1)$. Then $d$ divides
$(n^2+1)-(n+1)=n(n-1)$. Since $n$ is coprime to $n+1$, $d$ must
divide $n-1$, hence also must divide $(n+1)-(n-1)=2$. For $n$ even,
$d$ is obviously odd, so $d=1$. For $n$ odd, $n+1$ and $n^2+1$ are
both even, so $d=2$.
\end{proof}


\section{N\'eron models and component groups}\label{SecCG}

\subsection{Terminology and notation}
The notation in this section will be somewhat different from the
rest of the paper. Let $R$ be a complete discrete valuation ring,
with fraction field $K$ and algebraically  closed residue field $k$.

Let $A_K$ be an abelian variety over $K$. Denote by $A$ its N\'eron
model over $R$ and denote by $A_k^0$ the connected component of the
identity of the special fiber $A_k$ of $A$. There is an exact
sequence
$$
0\to A_k^0\to A_k\to \Phi_A\to 0,
$$
where $\Phi_A$ is a finite (abelian) group called the
\textit{component group of $A_K$}. We say that $A_K$ has
\textit{semi-abelian reduction} if $A_k^0$ is an extension of an
abelian variety $A_k'$ by an affine algebraic torus $T_A$ over $k$
(cf. \cite[p. 181]{NM}):
$$
0\to T_A\to A_k^0\to A_k'\to 0.
$$
We say that $A_K$ has \textit{toric reduction} if $A_k^0=T_A$. The
\textit{character group}
$$
M_A:=\Hom(T_{A}, \gm{k})
$$
is a free abelian group contravariantly associated to $A$.

Let $X_K$ be a smooth, proper, geometrically connected curve over
$K$. We say that $X$ is a \textit{semi-stable model} of $X_K$ over
$R$ if (cf. \cite[p. 245]{NM}):
\begin{enumerate}
\item[(i)] $X$ is a proper flat $R$-scheme.
\item[(ii)] The generic fiber of $X$ is $X_K$.
\item[(iii)] The special fiber $X_k$ is reduced, connected, and has
only ordinary double points as singularities.
\end{enumerate}
We will denote the set of irreducible components of $X_k$ by $C(X)$
and the set of singular points of $X_k$ by $S(X)$. Let $G(X)$ be the
\textit{dual graph of $X$}: The set of vertices of $G(X)$ is the set
$C(X)$, the set of edges is the set $S(X)$, the end points of an
edge $x$ are the two components containing $x$. Locally at $x\in
S(X)$ for the \'etale topology, $X$ is given by the equation
$uv=\pi^{m(x)}$, where $\pi$ is a uniformizer of $R$. The integer
$m(x)\geq 1$ is well-defined, and will be called the
\textit{thickness} of $x$. One obtains from $G(X)$ a graph with
length by assigning to each edge $x\in S(X)$ the length $m(x)$.

\subsection{Raynaud's theorem}\label{ssRaynaud} Let $X_K$ be a curve
over $K$ with semi-stable model $X$ over $R$. Let $J_K$ be the
Jacobian of $X_K$, let $J$ be the N\'eron model of $J_K$ over $R$, and
$\Phi:=J_k/J_k^0$. Let $\tX\to X$ be the minimal resolution of $X$.
Let $B(\tX)$ be the free abelian group generated by the elements of
$C(\tX)$. Let $B^0(\tX)$ be the kernel of the homomorphism
$$
B(\tX)\to \Z, \quad \sum_{C_i\in C(\tX)} n_i C_i\mapsto \sum n_i.
$$
The elements of $C(\tX)$ are Cartier divisors on $\tX$, hence for
any two of them, say $C$ and $C'$, we have an intersection number
$(C\cdot C')$. The image of the homomorphism
$$
\alpha: B(\tX)\to B(\tX), \quad C\mapsto \sum_{C'\in C(\tX)}(C\cdot
C')C'
$$
lies in $B^0(\tX)$. A theorem of Raynaud \cite[Thm. 9.6/1]{NM} says
that $\Phi$ is canonically isomorphic to $B^0(\tX)/\alpha(B(\tX))$.

The homomorphism $\phi: J_K(K)\to \Phi$ obtained from the
composition $$J_K(K)=J(R)\to J_k(k)\to \Phi$$ will be called the
\textit{canonical specialization map}. Let $D=\sum_Q n_Q Q$ be a
degree-$0$ divisor on $X_K$ whose support is in the set of
$K$-rational points. Let $P\in J_K(K)$ be the linear equivalence
class of $D$. The image $\phi(P)$ can be explicitly described as
follows. Since $X$ and $\tX$ are proper, $X(K)=X(R)=\tX(R)$. Since
$\tX$ is regular, each point $Q\in X(K)$ specializes to a unique
element $c(Q)$ of $C(\tX)$. With this notation, $\phi(P)$ is the
image of $\sum_Q n_Q c(Q)\in B^0(\tX)$ in $\Phi$.

\vspace{0.1in}

We apply Raynaud's theorem to compute $\Phi$ explicitly for a
special type of $X$. Assume that $X_k$ consists of two components
$Z$ and $Z'$ crossing transversally at $n\geq 2$ points $x_1,\dots,
x_n$. Denote $m_i:=m(x_i)$. Let $r:\tX\to X$ denote the resolution
morphism; it is a composition of blow-ups at the singular points. It
is well-known that $r^{-1}(x_i)$ is a chain of $m_i-1$ projective
lines. More precisely, the special fiber $\tX_k$ consists of $Z$ and
$Z'$ but now, instead of intersecting at $x_i$, these components are
joined by a chain $E_{1}, \cdots, E_{m_i-1}$ of projective lines,
where $E_{i}$ intersect $E_{i+1}$, $E_{1}$ intersects $Z$ at $x_i$
and $E_{m_i-1}$ intersects $Z'$ at $x_i$. All the singularities are
ordinary double points.

Assume $m_1=m_n=m\geq 1$ and $m_2=\cdots=m_{n-1}=1$ if $n\geq 3$.

If $m=1$, then $X=\tX$, so $B^0(\tX)$ is freely generated by
$z:=Z-Z'$. In this case Raynaud's theorem implies that $\Phi$ is
isomorphic to $B^0(\tX)$ modulo the relation $nz=0$.

If $m\geq 2$, let $E_1, \dots, E_{m-1}$ be the chain of projective
lines at $x_1$ and $G_1,\dots, G_{m-1}$ be the chain of projective
lines at $x_n$, with the convention that $Z$ in $\tX_k$ intersects
$E_1$ and $G_1$, cf. Figure \ref{Fig2}.
\begin{figure}
\begin{picture}(100,50)
\qbezier(10,45)(35,0)(50,25)\qbezier(50,25)(65,50)(90,5)
\qbezier(50,25)(65,0)(90,45)\qbezier(10,5)(35,50)(50,25)
\put(0,30){\line(1,1){15}}\put(0,20){\line(1,-1){15}}
\put(3,12.5){\line(0,1){25}}\put(85,45){\line(1,-1){15}}
\put(85,5){\line(1,1){15}}\put(97,12.5){\line(0,1){25}}
\put(7.5,13){$E_1$}\put(3.5,24){$E_2$}\put(7,34){$E_3$}
\put(86.5,13){$G_3$}\put(91,24){$G_2$}\put(88.5,34){$G_1$}
\put(37,34.5){$Z$}\put(60,34.5){$Z'$}
\end{picture}
\caption{$\tX_{k}$ for $n=5$ and $m=4$}\label{Fig2}
\end{figure}
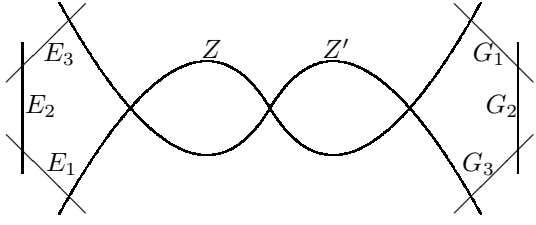
The elements $z:=Z-Z'$, $e_{i}:=E_{i}-Z'$, $g_i:=G_i-Z'$, $1\leq
i\leq m-1$ form a $\Z$-basis of $B^0(\tX)$. By Raynaud's theorem,
$\Phi$ is isomorphic to $B^0(\tX)$ modulo the following relations:

if $m=2$,
$$
-nz+e_1+g_1=0, \quad z-2e_1=0, \quad z-2g_1=0;
$$

if $m=3$,
$$
-nz+e_1+g_1=0, \quad z-2e_1+e_2=0, \quad z-2g_1+g_2=0,
$$
$$
e_{1}-2e_{2}=0, \quad g_{1}-2g_{2}=0;
$$

if $m\geq 4$
$$
-nz+e_1+g_1=0, \quad z-2e_1+e_2=0, \quad z-2g_1+g_2=0,
$$
$$
e_i-2e_{i+1}+e_{i+2}=0, \quad g_i-2g_{i+1}+g_{i+2}=0,\quad 1\leq
i\leq m-3,
$$
$$
e_{m-2}-2e_{m-1}=0, \quad g_{m-2}-2g_{m-1}=0.
$$

\begin{thm}\label{thmCG} Denote the images of $z$,
$e_i$, $g_i$ in $\Phi$ by the same letters, and let $\langle
z\rangle$ be the cyclic subgroup generated by $z$ in $\Phi$. Then
for any $n\geq 2$ and $m\geq 1$
\begin{enumerate}
\item[(i)] $\Phi\cong \Z/m(m(n-2)+2)\Z$.
\item[(ii)] If $m\geq 2$, then $\Phi$ is generated by $e_{m-1}$. Explicitly, for $1\leq i\leq
m-1$,
\begin{align*}
e_i&=(m-i)e_{m-1},\\ g_i&=\left(i(nm+1)-(2i-1)m\right)e_{m-1},\\
z&=me_{m-1}.
\end{align*}
\item[(iii)] $\Phi/\langle z\rangle\cong \Z/m\Z$.
\end{enumerate}
\end{thm}
\begin{proof} When $m=1$ the claim is obvious, so assume $m\geq 2$.
By \cite[Prop. 9.6/10]{NM}, $\Phi$ has order
$$
\sum_{i=1}^n \prod_{j\neq i}m_j= m^2(n-2)+2m.
$$
From the relations
\begin{align*}
e_{m-2}-2e_{m-1}&=0\\  e_i-2e_{i+1}+e_{i+2}&=0, \quad 1\leq i\leq m-3\\
z-2e_1+e_2& =0
\end{align*}
it follows inductively that $e_i=(m-i)e_{m-1}$ for $1\leq i\leq
m-1$, and $z=me_{m-1}$. Next, from the relations
$$
-nz+e_1+g_1=0\quad \text{and}\quad  z-2g_1+g_2=0
$$
we get $g_1=(nm-m+1)e_{m-1}$ and $g_2=(2nm-3m+2)e_{m-1}$. Finally,
if $m\geq 4$, the relations $g_i-2g_{i+1}+g_{i+2}=0$, $1\leq i\leq
m-3$, show inductively that
$$
g_i=\left(i(nm+1)-(2i-1)m\right)e_{m-1},\quad 1\leq i\leq m-1.
$$
This proves (i) and (ii), and (iii) is an immediate consequence of
(ii).
\end{proof}

\begin{rem}
Note that by the formula in Theorem \ref{thmCG}
$$
g_{m-1}=(m^2(n-2)+2m-(m(n-2)+1))e_{m-1}=-(m(n-2)+1)e_{m-1}.
$$
It is easy to see that $m(n-2)+1$ is coprime to the order
$m(m(n-2)+2)$ of $\Phi$. Hence $g_{m-1}$ is also a generator. This
is of course not surprising since the relations defining $\Phi$
remain the same if we interchange $e_i$'s and $g_i$'s.
\end{rem}

\subsection{Grothendieck's theorem} Grothendieck gave another
description of $\Phi$ in \cite{SGA7}. This description will be
useful for us when studying maps between the component groups
induced by isogenies of abelian varieties.

Let $A_K$ be an abelian variety over $K$ with semi-abelian
reduction. Denote by $\hat{A}_K$ the dual abelian variety of $A_K$.
As discussed in \cite{SGA7}, there is a non-degenerate pairing
$u_A:M_A\times M_{\hat{A}}\to \Z$ (called \textit{monodromy
pairing}) having nice functorial properties, which induces an exact
sequence
\begin{equation}\label{eqGrothM}
0\to M_{\hat{A}}\xrightarrow{u_A} \Hom(M_A, \Z)\to \Phi_A\to 0.
\end{equation}

Let $H\subset A_K(K)$ be a finite subgroup of order coprime to the
characteristic of $k$. Since $A(R)=A_K(K)$, $H$ extends to a
constant \'etale subgroup-scheme $\cH$ of $A$. The restriction to
the special fiber gives a natural injection $\cH_k\cong
H\hookrightarrow A_k(k)$, cf. \cite[Prop. 7.3/3]{NM}. Composing this
injection with $A_k\to \Phi_A$, we get the canonical homomorphism
$\phi:H\to \Phi_A$. Denote $H_0:=\ker(\phi)$ and
$H_1:=\mathrm{im}(\phi)$, so that there is a tautological exact
sequence
$$
0\to H_0\to H\xrightarrow{\phi} H_1\to 0.
$$
Let $B_K$ be the abelian variety obtained as the quotient of $A_K$
by $H$. Let $\varphi_K: A_K\to B_K$ denote the isogeny whose kernel
is $H$. By the N\'eron mapping property, $\varphi_K$ extends to a
morphism $\varphi: A\to B$ of the N\'eron models. On the special
fibers we get a homomorphism $\varphi_k: A_k\to B_k$, which induces
an isogeny $\varphi_k^0:A_k^0\to B_k^0$ and a homomorphism
$\varphi_\Phi: \Phi_A\to \Phi_B$. The isogeny $\varphi_k^0$
restricts to an isogeny $\varphi_t:T_A\to T_B$, which corresponds to
an injective homomorphisms of character groups $\varphi^\ast: M_B\to
M_A$ with finite cokernel.

\begin{thm}\label{propGroth} Assume $A_K$ has toric reduction.
There is an exact sequence
$$
0\to H_1\to \Phi_A\xrightarrow{\varphi_\Phi} \Phi_B\to H_0\to 0.
$$
\end{thm}
\begin{proof} The kernel of $\varphi_k$ is $\cH_k\cong H$. It is
clear that $\ker(\varphi_\Phi)=H_1$. Let $\hat{\varphi}_K:
\hat{B}_K\to \hat{A}_K$ be the isogeny dual to $\varphi_K$. Using
(\ref{eqGrothM}), one obtains a commutative diagram with exact rows
(cf. \cite[p. 8]{Ribet}):
$$
\xymatrix{ 0 \ar[r] & M_{\hat{A}} \ar[r]
\ar[d]^-{\hat{\varphi}^\ast} & \Hom(M_A, \Z)
\ar[d]^-{\Hom(\varphi^\ast, \Z)} \ar[r]
& \Phi_A \ar[d]^-{\varphi_\Phi} \ar[r] & 0 \\
0 \ar[r] & M_{\hat{B}} \ar[r] & \Hom(M_B, \Z) \ar[r] & \Phi_B \ar[r]
& 0.}
$$
From this diagram we get the exact sequence
$$
0\to \ker(\varphi_\Phi)\to
M_{\hat{B}}/\hat{\varphi}^\ast(M_{\hat{A}})\to
\Ext^1_\Z(M_A/\varphi^\ast(M_B), \Z)\to \coker(\varphi_\Phi)\to 0.
$$
Using the exact sequence $0\to \Z\to \Q\to \Q/\Z\to 0$, it is easy
to show that
$$
\Ext^1_\Z(M_A/\varphi^\ast(M_B), \Z)\cong
\Hom(M_A/\varphi^\ast(M_B), \Q/\Z)=: (M_A/\varphi^\ast(M_B))^\vee,
$$
so there is an exact sequence of abelian groups
\begin{equation}\label{eq2}
0\to \ker(\varphi_\Phi)\to
M_{\hat{B}}/\hat{\varphi}^\ast(M_{\hat{A}})\to
(M_A/\varphi^\ast(M_B))^\vee\to \coker(\varphi_\Phi)\to 0.
\end{equation}
So far we have not used the assumption that $A_K$ has toric
reduction. Under this assumption, $B_K$ also has toric reduction,
and $H_0$ is the kernel of $\varphi_t: T_A\to T_B$. Hence
$(M_A/\varphi^\ast(M_B))^\vee\cong H_0$. Next, \cite[Thm. 8.6]{CS}
implies that $M_{\hat{B}}/\hat{\varphi}^\ast(M_{\hat{A}})\cong H_1$.
Thus, we can rewrite (\ref{eq2}) as
$$
0\to \ker(\varphi_\Phi)\to H_1\to H_0\to \coker(\varphi_\Phi)\to 0.
$$
Since $\ker(\varphi_\Phi)= H_1$, this implies that
$\coker(\varphi_\Phi)\cong H_0$.
\end{proof}

\section{Component groups of $J_0(xy)$}

\subsection{Component groups at $x$ and $y$} We return to the
notation in Section \ref{SecCDG}. As we mentioned in
$\S$\ref{ssDMC}, $X_0(xy)$ is smooth over $A[1/xy]$.

\begin{prop}\label{propModelxy}\hfill
\begin{itemize}
\item[(i)] $X_0(xy)_{F_x}$ has a semi-stable model over $\cO_x$ such that $X_0(xy)_{\F_x}$
consists of two irreducible components both isomorphic to
$X_0(y)_{\F_x}\cong \P^1_{\F_q}$ intersecting transversally in $q+1$
points. Two of these singular points have thickness $q+1$, and the
other $q-1$ points have thickness $1$.
\item[(ii)] $X_0(xy)_{F_y}$ has a semi-stable model over $\cO_y$ such that $X_0(xy)_{\F_y}$
consists of two irreducible components both isomorphic to
$X_0(x)_{\F_y}\cong \P^1_{\F_{q^2}}$ intersecting transversally in
$q+1$ points. All these singular points have thickness $1$.
\end{itemize}
\end{prop}
\begin{proof}
The fact that $X_0(xy)_F$ has a model over $\cO_x$ and $\cO_y$ with
special fibers of the stated form follows from the same argument as
in the case of $X_0(v)_{F}$ over $\cO_v$ ($v\in |F|-\infty$)
discussed in \cite[$\S$5]{GekelerUber}. We only clarify why the
number of singular points and their thickness are as stated.

(i) The special fiber $X_0(xy)_{\F_x}$ consists of two copies of $X_0(y)_{\F_x}$.
The set of points $Y_0(y)(\overline{\F}_x)$ is in bijection with the isomorphism
classes of pairs $(\phi, C_y)$, where $\phi$ is a rank-$2$ Drinfeld
$A$-module over $\overline{\F}_x$ and $C_y\cong A/y$ is a cyclic
subgroup of $\phi$. The two copies of $X_0(y)_{\F_x}$ intersect
exactly at the points corresponding to $(\phi, C_y)$ with $\phi$
supersingular; more precisely, $(\phi, C_y)$ on the first copy is
identified with $(\phi^{(x)}, C_y^{(x)})$ on the second copy where
$\phi^{(x)}$ is the image of $\phi$ under the Frobenius isogeny and
$C_y^{(x)}$ is subgroup of $\phi^{(x)}$ which is the image of $C_y$,
cf. \cite{GekelerUber}.

Now, by Lemma \ref{lemPrelim1}, up to an isomorphism over $\overline{\F}_x$, 
there is a unique supersingular
Drinfeld module $\phi$ in characteristic $x$ and $j(\phi)=0$. It is
easy to see that $\phi$ has $q_y+1=q^2+1$ cyclic subgroups
isomorphic to $A/y$, so the set $S=\{(\phi, C_y)\ |\ C_y\subset
\phi[y]\}$ has cardinality $q^2+1$. By Lemma \ref{lemPrelim2},
$\Aut(\phi)\cong \F_{q^2}^\times$. This group naturally acts $S$,
and the orbits are in bijection with the singular points of
$X_0(xy)_{\F_x}$. Since the genus of $X_0(xy)_F$ is $q$, the
arithmetic genus of $X_0(xy)_{\F_x}$ is also $q$ due to the flatness
of $X_0(xy)\to \Spec(A)$; see \cite[Cor. III.9.10]{Hartshorne}.
Using the fact that the genus of $X_0(y)_F$ is zero, a simple
calculation shows that the number of singular points of
$X_0(xy)_{\F_x}$ is $q+1$, cf. \cite[p. 298]{Hartshorne}. Next, by
Lemma \ref{lemPrelim3}, the stabilizer in $\Aut(\phi)$ of $(\phi,
C_y)$ is either $\F_q^\times$ or $\F_{q^2}^\times$. Let $s$ be the
number of pairs $(\phi, C_y)$ with stabilizer $\F_{q^2}^\times$. Let
$t$ be the number of orbits of pairs with stabilizers $\F_q^\times$;
each such orbit consists of $\#(\F_{q^2}^\times/\F_q^\times)=q+1$
pairs $(\phi, C_y)$. Hence we have
$$
(q+1)t+s=q^2+1 \quad \text{and}\quad t+s=q+1.
$$
This implies that $t=q-1$ and $s=2$. Finally, as is explained in
\cite{GekelerUber}, the thickness of the singular point
corresponding to an isomorphism class of $(\phi, C_y)$ is equal to
$\#(\Aut(\phi, C_y)/\F_q^\times)$.

(ii) Similar to the previous case, $X_0(xy)_{\F_y}$ consists of two
copies of $X_0(x)_{\F_y}\cong \P^1_{\F_{q^2}}$. The two copies of
$X_0(x)_{\F_y}$ intersect exactly at the points corresponding to the
isomorphism classes of pairs $(\phi, C_x)$ with $\phi$
supersingular. Again by Lemma \ref{lemPrelim1}, up to an isomorphism over $\overline{\F}_y$, there is a unique 
supersingular $\phi$ and $j(\phi)\neq 0$. Hence, by Lemma
\ref{lemPrelim3}, $\Aut(\phi, C_x)\cong \F_q^\times$ for any $C_x$.
There are $q_x+1=q+1$ cyclic subgroups in $\phi$ isomorphic to
$A/x$. The rest of the argument is the same as in the previous case.
\end{proof}

\begin{thm}\label{thmCGxy} Let $\Phi_v$ denote the group of connected components
of $J_0(xy)$ at $v\in |F|$. Let $Z$ and $Z'$ be the
irreducible components in Proposition \ref{propModelxy} with the
convention that the reduction of $[\infty]$ lies on $Z'$. Let
$z=Z-Z'$.
\begin{itemize}
\item[(i)] $\Phi_x\cong \Z/(q^2+1)(q+1)\Z$.
\item[(ii)]
$\Phi_y\cong\Z/(q+1)\Z$.
\item[(iii)] Under the canonical specialization
map $\phi_x:\cC\to \Phi_x$ we have 
$$\phi_x(c_x)=0\quad \text{and} \quad \phi_x(c_y)=z.$$ In particular, $q^2+1$ divides the
order of $c_y$.
\item[(iv)] Under the canonical specialization map
$\phi_y:\cC\to \Phi_y$ we have 
$$\phi_y(c_x)=z\quad \text{and} \quad \phi_y(c_y)=0.
$$ In particular, $q+1$
divides the order of $c_x$.
\end{itemize}
\end{thm}
\begin{proof} (i) and (ii) follow from
Theorem \ref{thmCG} and Proposition \ref{propModelxy}.

(iii) The cusps reduce to distinct points in the smooth locus of
$X_0(xy)_{\F_x}$, cf. \cite{vdPT}. Since by Theorem \ref{thmCG} we
know that $z$ has order $q^2+1$ in the component group $\Phi_x$, it
is enough to show that the reductions of $[y]$ and $[\infty]$ lie on
distinct components $Z$ and $Z'$ in $X_0(xy)_{\F_x}$, but the
reductions of $[x]$ and $[\infty]$ lie on the same component. The
involution $W_x$ interchanges the two components $X_0(y)_{\F_x}$,
cf. \cite[(5.3)]{GekelerUber}. Since $W_x([\infty])=[y]$, the
reductions of $[\infty]$ and $[y]$ lie on distinct components.
On the other hand, $W_y$ acts on $X_0(xy)_{\F_y}$ by
acting on each component $X_0(y)_{\F_x}$ separately, without
interchanging them. Since $W_y([\infty])=[x]$, the reductions of
$[\infty]$ and $[x]$ lie on the same component. 

(iv) The argument is similar to (iii). Here $W_y$ interchanges the two
components $X_0(x)_{\F_y}$ of $X_0(xy)_{\F_y}$ and $W_x$ maps the
components to themselves. Hence $[\infty]$ and $[y]$ lie on one
component and $[0]$ and $[x]$ on the other component. 
\end{proof}

\begin{thm}\label{thmCDG} The cuspidal divisor group 
$$
\cC\cong \Z/(q+1)\Z\oplus \Z/(q^2+1)\Z
$$
is the direct sum of the cyclic subgroups generated by $c_x$ and $c_y$, 
which have orders $(q+1)$ and $(q^2+1)$, respectively. 
$($Note that $\cC$ is cyclic if $q$ is even, but it is not cyclic if $q$ is odd. $)$
\end{thm}
\begin{proof} By Lemma
\ref{lemFApprx} and Theorem \ref{thmCGxy}, $\cC$ is generated by $c_x$ and $c_y$, which have 
orders $(q+1)$ and $(q^2+1)$, respectively. If the subgroup  of $\cC$
generated by $c_x$ non-trivially intersects with the subgroup generated 
by $c_y$, then, by Lemma \ref{lem_gcd}, $q$ must be odd and $\frac{q+1}{2}c_x=\frac{q^2+1}{2}c_y$. 
Applying $\phi_y$ to both sides of this equality, 
we get $\frac{q+1}{2}z=0$, which is a contradiction since $z$ generates 
$\Phi_y\cong \Z/(q+1)\Z$. 
\end{proof}

\begin{rem}
The divisor class $c_0$ has order $(q+1)(q^2+1)$ (resp. $(q+1)(q^2+1)/2$) 
if $q$ is even (resp. odd). 
\end{rem}


\subsection{Component group at $\infty$}\label{ssInfty} To obtain a model of
$X_0(xy)_{\Fi}$ over $\cO_\infty$, instead of relying on the moduli
interpretation of $X_0(xy)$, one has to use the existence of
analytic uniformization for this curve; see \cite[$\S$4.2]{PapAIF}.
As far as the structure of the special fiber $X_0(xy)_{\F_\infty}$
is concerned, it is more natural to compute the dual graph of
$X_0(xy)_{\F_\infty}$ directly using the quotient $\G_0(xy)\bs\cT$
of the Bruhat-Tits tree $\cT$ of $\PGL_2(\Fi)$. For the definition
of $\cT$, and more generally for the basic theory of trees and
groups acting on trees, we refer to \cite{SerreT}.

The quotient graph $\G_0(xy)\bs\cT$ was first computed by Gekeler
\cite[(5.2)]{GekelerKF}. For our purposes we will need to know the
relative position of the cusps on $\G_0(xy)\bs\cT$ and also the
stabilizers of the edges. To obtain this more detailed information,
and for the general sake of completeness, we recompute
$\G_0(xy)\bs\cT$ in this subsection using the method in \cite{GN}.

Denote
$$
G_0=\GL_2(\F_q)
$$
and
$$
G_i=\left\{\begin{pmatrix} a & b \\ 0 & d\end{pmatrix}\in \GL_2(A)\
| \ \deg(b)\leq i\right\}, \quad i\geq 1.
$$
As is explained in \cite{GN}, $\G_0(xy)\bs \cT$ can be constructed
in ``layers'', where the vertices of the $i$th layer (in \cite{GN}
called \textit{type-$i$ vertices}) are the orbits
$$
X_i:= G_i\bs \P^1(A/xy)
$$
and the edges connecting type-$i$ vertices to type-$(i+1)$ vertices, called \textit{type-$i$ edges}, 
are the orbits
$$
Y_i:= (G_i\cap G_{i+1})\bs \P^1(A/xy).
$$
There are obvious maps $Y_i\to X_i$, $Y_i\to X_{i+1}$ and $X_i\to X_{i+1}$ which are used to define the adjacencies 
of vertices in $X_i$ and $X_{i+1}$; see \cite[1.7]{GN}. The graph $\G_0(xy)\bs \cT$ is isomorphic 
to the graph with set of vertices $\bigsqcup_{i\geq 0} X_i$ and set of edges $\bigsqcup_{i\geq 0} Y_i$ with 
the adjacencies defined by these maps. 

Note that $\P^1(A/xy)=\P^1(\F_x)\times \P^1(\F_y)$. We will
represent the elements of $\P^1(A/xy)$ as couples $[P; Q]$ where
$P\in \P^1(\F_x)$ and $Q\in \P^1(\F_y)$. With this notation, $G_i$
acts diagonally on $[P; Q]$ via its images in $\GL_2(\F_x)$ and
$\GL_2(\F_y)$, respectively.

The group $G_0$ acting on $\P^1(A/xy)$ has $3$ orbits, whose representatives
are
$$
[\begin{pmatrix} 1\\ 0\end{pmatrix}; \begin{pmatrix} 1\\
0\end{pmatrix}],\quad [\begin{pmatrix} 1\\ 0\end{pmatrix};
\begin{pmatrix} 0\\ 1\end{pmatrix}],\quad [\begin{pmatrix} 1\\ 0\end{pmatrix}; \begin{pmatrix} x\\ 1\end{pmatrix}],
$$
where in the last element we write $x$ for the image in $\F_y$ of
the monic generator of $x$ under the canonical homomorphism $A\to
A/y$. The orbit of $[\begin{pmatrix} 1\\
0\end{pmatrix}; \begin{pmatrix} 1\\ 0\end{pmatrix}]$ has length
$q+1$, the orbit of $[\begin{pmatrix} 1\\ 0\end{pmatrix};
\begin{pmatrix} 0\\ 1\end{pmatrix}]$
has length $q(q+1)$, and the orbit of $[\begin{pmatrix} 1\\
0\end{pmatrix}; \begin{pmatrix} x\\ 1\end{pmatrix}]$ has length
$q(q^2-1)$, cf. \cite[Prop. 2.10]{GN}.
Next, note that $G_0\cap G_1$ is the subgroup $B$ of the
upper-triangular matrices in $\GL_2(\F_q)$. The $G_0$-orbit of
$[\begin{pmatrix} 1\\ 0\end{pmatrix}; \begin{pmatrix} 1\\
0\end{pmatrix}]$ splits into two $B$-orbits with representatives:
\begin{equation}\label{eqCR1}
[\begin{pmatrix} 1\\ 0\end{pmatrix}; \begin{pmatrix} 1\\
0\end{pmatrix}]\quad\text{and}\quad [\begin{pmatrix} 0\\
1\end{pmatrix}; \begin{pmatrix} 0\\ 1\end{pmatrix}].
\end{equation}
The lengths of these $B$-orbits are $1$ and $q$, respectively. The
$G_0$-orbit of $[\begin{pmatrix} 1\\ 0\end{pmatrix}; \begin{pmatrix}
0\\ 1\end{pmatrix}]$ splits into three $B$-orbits with
representatives:
\begin{equation}\label{eqCR2}
[\begin{pmatrix} 1\\ 0\end{pmatrix}; \begin{pmatrix} 0\\
1\end{pmatrix}],\quad [\begin{pmatrix} 0\\ 1\end{pmatrix};
\begin{pmatrix} 1\\ 0\end{pmatrix}], \quad [\begin{pmatrix} 1\\ 1\end{pmatrix}; \begin{pmatrix} 0\\ 1\end{pmatrix}].
\end{equation}
The lengths of these $B$-orbits are $q$, $q$, $q(q-1)$,
respectively. Finally, the $G_0$-orbit of $[\begin{pmatrix} 1\\
0\end{pmatrix}; \begin{pmatrix} x\\ 1\end{pmatrix}]$ splits into
$(q+1)$ $B$-orbits each of length $q(q-1)$. The previous statements
can be deduced from Proposition 2.11 in \cite{GN}.
It turns out that the elements of $\P^1(\F_x)\times \P^1(\F_y)$
listed in (\ref{eqCR1}) and (\ref{eqCR2}) combined form a complete
set of $G_1$-orbit representatives. For $i\geq 1$, the set of
$G_i$-orbit representatives obviously contains a complete set of
$G_{i+1}$-orbit representatives. A small calculation shows that
\begin{equation}\label{eqCR3}
[\begin{pmatrix} 1\\ 0\end{pmatrix}; \begin{pmatrix} 1\\
0\end{pmatrix}],\quad [\begin{pmatrix} 0\\ 1\end{pmatrix}; \begin{pmatrix} 0\\ 1\end{pmatrix}],\quad [\begin{pmatrix} 1\\
0\end{pmatrix}; \begin{pmatrix} 0\\ 1\end{pmatrix}],\quad [\begin{pmatrix} 0\\ 1\end{pmatrix}; \begin{pmatrix} 1\\
0\end{pmatrix}]
\end{equation}
is a complete set of $G_i$-orbit representatives for any $i\geq 2$.
Moreover, the elements $[\begin{pmatrix} 1\\ 1\end{pmatrix};
\begin{pmatrix} 0\\ 1\end{pmatrix}]$
and $[\begin{pmatrix} 0\\ 1\end{pmatrix}; \begin{pmatrix} 0\\
1\end{pmatrix}]$ are in the same $G_2$-orbit. We recognize the
elements in (\ref{eqCR3}) as the cusps $[\infty], [0], [x], [y]$,
respectively. 
Overall, the structure of $\G_0(xy)\bs \cT$ is described by the
diagram in Figure \ref{FigDiagram}. In the diagram the broken line $- - -$ 
indicates that there are $(q-1)$ distinct edges joining the
corresponding vertices, and an arrow $\rightarrow$ indicates an infinite
half-line.
\begin{figure}
$$
\xymatrix{
& {[\begin{pmatrix} 1\\ 0\end{pmatrix}; \begin{pmatrix} 1\\ 0\end{pmatrix}]} \ar@{-}[r] & {[\begin{pmatrix} 1\\ 0\end{pmatrix}; \begin{pmatrix} 1\\ 0\end{pmatrix}]}\ar[rr]^-{[\infty]} & & \\
{[\begin{pmatrix} 1\\ 0\end{pmatrix}; \begin{pmatrix} 1\\
0\end{pmatrix}]} \ar@{-}[r] \ar@{-}[ur] & {[\begin{pmatrix} 0\\
1\end{pmatrix}; \begin{pmatrix} 0\\ 1\end{pmatrix}]}
\ar@{~}[r]& {[\begin{pmatrix} 0\\ 1\end{pmatrix}; \begin{pmatrix} 0\\ 1\end{pmatrix}]}\ar[rr]^-{[0]} & & \\
{[\begin{pmatrix} 1\\ 0\end{pmatrix}; \begin{pmatrix} x\\
1\end{pmatrix}]} \ar@{--}[r] \ar@{-}[ur] \ar@{-}[dr] &
{[\begin{pmatrix} 1\\ 1\end{pmatrix};
\begin{pmatrix} 0\\ 1\end{pmatrix}]} \ar@{-}[ur] &  & \\
{[\begin{pmatrix} 1\\ 0\end{pmatrix}; \begin{pmatrix} 0\\
1\end{pmatrix}]} \ar@{~}[r] \ar@{-}[ur] \ar@{-}[dr] &
{[\begin{pmatrix} 1\\ 0\end{pmatrix}; \begin{pmatrix} 0\\
1\end{pmatrix}]} \ar@{-}[r]& {[\begin{pmatrix} 1\\ 0\end{pmatrix};
\begin{pmatrix} 0\\ 1\end{pmatrix}]} \ar[rr]^-{[x]} & & \\
& {[\begin{pmatrix} 0\\ 1\end{pmatrix}; \begin{pmatrix} 1\\ 0\end{pmatrix}]} \ar@{-}[r]& {[\begin{pmatrix} 0\\ 1\end{pmatrix}; \begin{pmatrix} 1\\ 0\end{pmatrix}]} \ar[rr]^-{[y]} & & \\
X_0 & X_1 & X_2}
$$
\caption{$\G_0(xy)\bs \cT$}\label{FigDiagram}
\end{figure}

Now we compute the stabilizers of the edges. 
Let $e$ be an edge in $\G_0(xy)\bs \cT$ of type $i$. Let $$O(e)=(G_i\cap G_{i+1})[P;Q]$$ be its corresponding orbit 
in $(G_i\cap G_{i+1})\bs \P^1(A/xy)$. Then for a preimage $\tilde{e}$ of $e$ in $\cT$ we have 
$$
\# \Stab_{\G_0(xy)}(\tilde{e})=\# \Stab_{G_i\cap G_{i+1}}([P;Q])=\frac{\# (G_i\cap G_{i+1})}{\# O(e)}.
$$
Using this observation, we conclude from our previous discussion that the edges connecting  
$[\begin{pmatrix} 1\\ 0\end{pmatrix}; \begin{pmatrix} x\\ 1\end{pmatrix}]\in X_0$ to any vertex in $X_1$ have 
preimages whose stabilizers have order $\# B/q(q-1)=q-1$. 
The preimages of the edges connecting $[\begin{pmatrix} 1\\ 0\end{pmatrix}; \begin{pmatrix} 0\\ 1\end{pmatrix}]\in X_0$ 
to $[\begin{pmatrix} 1\\ 1\end{pmatrix}; \begin{pmatrix} 0\\ 1\end{pmatrix}]\in X_1$ and 
$[\begin{pmatrix} 1\\ 0\end{pmatrix}; \begin{pmatrix} 0\\ 1\end{pmatrix}]\in X_1$ have stabilizers of 
orders $q-1$ and $(q-1)^2$, respectively. 
(Note that if a stabilizer has order $(q-1)$ then it is equal to the center $Z(\G_0(xy))\cong \F_q^\times$ of $\G_0(xy)$, 
as the center is a subgroup of any stabilizer.) The \textit{valency} of a vertex $v$ in a graph without loops 
is the number of distinct edges having $v$ as an endpoint. (A
loop is an edge whose endpoints are the same.) 
Consider the vertex $v=[\begin{pmatrix} 1\\ 1\end{pmatrix}; \begin{pmatrix} 0\\ 1\end{pmatrix}]\in X_1$. Its 
valency is $(q+1)$. Let $\tilde{v}$ be a preimage of $v$ in $\cT$. Since the valency of $\tilde{v}$ is also 
$q+1$, $\Stab_{\G_0(xy)}(\tilde{v})$ acts trivially on all edges having $\tilde{v}$ as an endpoint. Hence 
the stabilizer of any such edge is equal to $\Stab_{\G_0(xy)}(\tilde{v})$. We already determined that the 
stabilizer of a preimage of an edge connecting $v$ to a type-$0$ vertex is $\F_q^\times$. This 
implies that the stabilizer in $\G_0(xy)$ of a preimage of the edge connecting $v$ to 
$[\begin{pmatrix} 0\\ 1\end{pmatrix}; \begin{pmatrix} 0\\ 1\end{pmatrix}]\in X_2$ is also $\F_q^\times$. 
Finally, consider the vertex $w=[\begin{pmatrix} 0\\ 1\end{pmatrix}; \begin{pmatrix} 0\\ 1\end{pmatrix}]\in X_1$. 
Its valency is $3$. 
Let $S, S_1, S_2, S_3$ be the orders of stabilizers in $\G_0(xy)$ of a preimage $\tilde{w}$ of $w$ in $\cT$, 
and the edges connecting $w$ to 
$[\begin{pmatrix} 1\\ 0\end{pmatrix}; \begin{pmatrix} 1\\ 0\end{pmatrix}]\in X_0$, 
$[\begin{pmatrix} 1\\ 0\end{pmatrix}; \begin{pmatrix} x\\ 1\end{pmatrix}]\in X_0$, 
$[\begin{pmatrix} 0\\ 1\end{pmatrix}; \begin{pmatrix} 0\\ 1\end{pmatrix}]\in X_2$, respectively. 
From our discussion of the lengths of orbits of type-$0$ edges, we have 
$S_1=(q-1)^2$ and $S_2=(q-1)$. Obviously, $S_i$'s divide $S$. On the other hand, counting the lengths of orbits 
of $\Stab_{\G_0(xy)}(\tilde{w})$ acting on the set of (non-oriented) 
edges in $\cT$ having $\tilde{w}$ as an endpoint, 
we get 
$$
q+1=\frac{S}{S_1}+\frac{S}{S_2} +\frac{S}{S_3}=\frac{S}{(q-1)^2}+\frac{S}{(q-1)} +\frac{S}{S_3}.  
$$
This implies $S=S_3=(q-1)^2$.  To summarize, in 
Figure \ref{FigDiagram} a wavy line $\sim$ indicates that a preimage
of the corresponding edge in $\cT$ has a stabilizer in $\G_0(xy)$ of
order $(q-1)^2$. The edges connecting $[\begin{pmatrix}
1\\ 0\end{pmatrix}; \begin{pmatrix} x\\ 1\end{pmatrix}]$ or $[\begin{pmatrix} 1\\ 1\end{pmatrix}; \begin{pmatrix} 0\\
1\end{pmatrix}]$ to any other vertex have preimages in $\cT$ whose
stabilizers in $\G_0(xy)$ are isomorphic to $\F_q^\times$ .

\vspace{0.1in}

Now from \cite[$\S$4.2]{PapAIF} one deduces the following.
The quotient graph $\G_0(xy)\bs \cT$, without the infinite half-lines, is the dual
graph of the special fiber of a semi-stable model of $X_0(xy)_{\Fi}$
over $\Spec(\cO_\infty)$. The special fiber $X_0(xy)_{\F_\infty}$ has $6$ irreducible
components $Z, Z', E, E', G, G'$, all isomorphic to $\P^1_{\F_q}$,
such that $Z$ and $Z'$ intersect in $q-1$ points, $E$ intersects $Z$
and $E'$, $E'$ intersects $Z'$ and $E$, $G$ intersects $Z$ and $G'$,
$G'$ intersects $Z'$ and $G$.  Moreover, all intersection points are ordinary double singularities. 
By \cite[Prop. 4.3]{PapAIF}, the thickness of the singular point corresponding to 
an edge $e\in \G_0(xy)\bs\cT$ is 
$$
\#(\Stab_{\G_0(xy)}(\tilde{e})/\F_q^\times),
$$ 
hence all intersection points on $Z$ or
$Z'$ have thickness $1$, but the intersection points of $E$ and $E'$,
and of $G$ and $G'$ have thickness $(q-1)$, cf. Figure \ref{Fig1}.
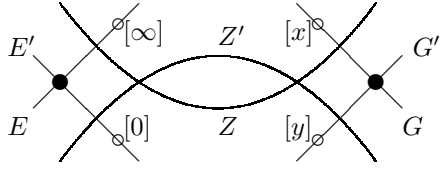
\begin{figure}
\begin{picture}(80,50)
\qbezier(10,10)(40,50)(70,10)\qbezier(10,40)(40,0)(70,40)
\put(5,20){\line(1,1){20}}\put(5,30){\line(1,-1){20}}
\put(55,10){\line(1,1){20}}\put(55,40){\line(1,-1){20}}
\put(10,25){\circle*{3}}\put(70,25){\circle*{3}}
\put(40,15){$Z$}\put(40,32){$Z'$}\put(0,15){$E$}\put(0,30){$E'$}
\put(75,15){$G$}\put(77,30){$G'$}
\put(21,36){\circle{2}}\put(22,33){$[\infty]$}
\put(21,14){\circle{2}}\put(22,15){$[0]$}
\put(59,36){\circle{2}}\put(52.5,33){$[x]$}\put(59,14){\circle{2}}\put(52.5,15){$[y]$}
\end{picture}
\caption{$X_0(xy)_{\F_\infty}$ for $q=3$}\label{Fig1}
\end{figure}
From the structure of $\G_0(xy)\bs\cT$, one also concludes that the
reductions of the cusps are smooth points in $X_0(xy)_{\F_\infty}$.
Moreover, $[\infty]$, $[0]$, $[x]$, $[y]$ reduce to points on $E$,
$E'$, $G$, $G'$ respectively.

Blowing up $X_0(xy)_{\cO_\infty}$ at the intersection points of
$E$, $E'$, and $G$, $G'$, $(q-2)$-times each, we obtain the minimal
regular model of $X_0(xy)_F$ over $\Spec(\cO_\infty)$. This is a
curve of the type discussed in $\S$\ref{ssRaynaud} with $m=n=(q+1)$,
and we enumerate its irreducible components so that $E_1=E$,
$E_q=E'$, $G_1=G$, $G_q=G'$.

\begin{thm}\label{thmCGinfty} Let $\phi_\infty:\cC\to \Phi_\infty$ denote the
canonical specialization map.
\begin{enumerate}
\item[(i)] $\Phi_\infty\cong \Z/(q^2+1)(q+1)\Z$.
\item[(ii)] $\phi_\infty(c_x)=(q^2+1)e_q$ and $\phi_\infty(c_y)=-q(q+1)e_q=(q^3+1)e_q$. 
\item[(iii)] If $q$ is even, then $\phi_\infty:\cC\xrightarrow{\sim} \Phi_\infty$ 
is an isomorphism. 
\item[(iv)]If $q$ is odd, then there is an
exact sequence
$$
0\to \Z/2\Z\to \cC\xrightarrow{\phi_\infty}\Phi_\infty\to \Z/2\Z\to
0.
$$
\end{enumerate}
\end{thm}
\begin{proof}
Part (i) is an immediate consequence of the preceding discussion and
Theorem \ref{thmCG}. We have determined the reductions of the
cusps at $\infty$, so using Theorem \ref{thmCG}, we get 
$$
\phi_\infty(c_x)=g_1-e_1=(q^2+q+1)e_q-qe_q=(q^2+1)e_q
$$
and
$$
\phi_\infty(c_y)=g_q-e_1=-q^2e_q-qe_q=-q(q+1)e_q,
$$
which proves (ii).  Since $\mathrm{gcd}(q^2+1, q(q+1))=1$ (resp. $2$) if 
$q$ is even (resp. odd), cf. Lemma \ref{lem_gcd}, the subgroup of $\Phi_\infty$ generated by
$\phi_\infty(c_x)$ and $\phi_\infty(c_y)$ is $\langle e_q\rangle$ (resp. 
$\langle 2e_q\rangle$) if $q$ is even (resp. odd). On the other hand, 
we know that $e_q$ generates $\Phi_\infty$. 
Therefore, if $q$ is even, then 
$\phi_\infty$ is surjective, and if $q$ is odd, then the 
cokernel of $\phi_\infty$ is isomorphic to $\Z/2\Z$. 
The claims (iii) and (iv) now follow from Theorem \ref{thmCDG}.
\end{proof}

\begin{rem}
We note that (iii) and a slightly weaker version of (iv) in Theorem
\ref{thmCGinfty} can be deduced from Theorem \ref{thmCDG} and a
result of Gekeler \cite{GekelerDoc}. In fact, in \cite[p.
366]{GekelerDoc} it is proven that for an arbitrary $\fn$ the kernel
of the canonical homomorphism from the cuspidal divisor group of
$X_0(\fn)_F$ to $\Phi_\infty$ is a quotient of $(\Z/(q-1)\Z)^{c-1}$,
where $c$ is the number of cusps of $X_0(\fn)_F$. In our case, this
result says that $\ker(\phi_\infty)$ is a quotient of
$(\Z/(q-1)\Z)^3$. Now suppose $q$ is even. Then $\cC\cong
\Z/(q^2+1)(q+1)\Z$. Since for even $q$, $\mathrm{gcd}(q-1,
(q^2+1)(q+1))=1$, $\phi_\infty$ must be injective. But by (i),
$\#\Phi_\infty=(q^2+1)(q+1)=\# \cC$, so $\phi_\infty$ is also
surjective. When $q=2$, the fact that $\#\Phi_\infty=15$ and
$\phi_\infty$ is an isomorphism is already contained in
\cite[(5.3.1)]{GekelerDoc}.

Now suppose $q$ is odd. Then $\cC\cong \Z/(q^2+1)\Z\oplus
\Z/(q+1)\Z$. Since $$\mathrm{gcd}(q-1, q+1)=\mathrm{gcd}(q-1,
q^2+1)=2,$$ $\ker(\phi_\infty)\subset (\Z/2\Z)^2$. Since
$\Phi_\infty$ is cyclic but $\cC$ is not, $\ker(\phi_\infty)$ is not
trivial, hence it is either $\Z/2\Z$ or $(\Z/2\Z)^2$. (Theorem
\ref{thmCGinfty} implies that the second possibility does not
occur.)
\end{rem}

\begin{notn}
Let $\cC_0$ be the subgroup of $\cC$ generated by $c_y$. 
\end{notn}

\begin{cor}\label{lemC0} The cyclic group $\cC_0$ has order $q^2+1$. Under the 
canonical specializations $\cC_0$ maps injectively into $\Phi_x$ and $\Phi_\infty$, and 
$\cC_0$ is the kernel of $\phi_y$. 
\end{cor}
\begin{proof}
The claims easily follow from Theorems
\ref{thmCGxy}, \ref{thmCDG} and \ref{thmCGinfty}.
\end{proof}


\section{Component groups of $J^{xy}$}\label{SecCG-J'}


\subsection{A class number formula}\label{ssCN}
Let $H$ be a quaternion algebra over $F$. Let $\Ram\subset |F|$ be
the set of places where $H$ ramifies. Assume $\infty\in \Ram$.
Denote $\cR=\Ram-\infty$. Note that $\cR\neq \emptyset$ since
$\#\Ram$ is even.

Let $\Theta$ be a hereditary $A$-order in $H$. Let $I_1, \dots, I_h$
be the isomorphism classes of left $\Theta$-ideals. It is known that
$h(\Theta):=h$, called the \textit{class number} of $\Theta$, is
finite. For $i=1,\dots, h$ we denote by $\Theta_i$ the right order
of the respective $I_i$. (For the definitions see \cite{Vigneras}.)
Denote
$$
M(\Theta)=\sum_{i=1}^h (\Theta_i^\times:\F_q^\times)^{-1}.
$$
It is not hard to show that each $\Theta_i^\times$ is isomorphic to
either $\F_q^\times$ or $\F_{q^2}^\times$; see \cite[p. 383]{DvG}.
Let $U(\Theta)$ be the number of right orders $\Theta_i$ such that
$\Theta_i^\times\cong \F_{q^2}^\times$. In particular,
$$
h(\Theta)=M(\Theta)+U(\Theta)\left(1-\frac{1}{q+1}\right).
$$
\begin{defn}
For a subset $S$ of $|F|$, let
$$
\Odd(S)=\left\{
         \begin{array}{ll}
           1, & \hbox{if all places in $S$ have odd degrees;} \\
           0, & \hbox{otherwise.}
         \end{array}
       \right.
$$
\end{defn}
Let $\cS\subset |F|-\infty$ be a finite (possibly empty) set of
places such that $\cR\cap \cS=\emptyset$. Let $\fn\lhd A$ be the
square-free ideal whose support is $\cS$. Let $\Theta$ be an Eichler
$A$-order of level $\fn$. (When $\cS=\emptyset$, $\Theta$ is a
maximal $A$-order in $H$.) The formulae that follow are special
cases of (1), (4) and (6) in \cite{DvG}:
$$
M^\cS(H):=M(\Theta)=\frac{1}{q^2-1}\prod_{v\in
\cR}(q_v-1)\prod_{w\in \cS}(q_w+1),
$$
$$
U^\cS(H):=U(\Theta)=2^{\# \cR+\#\cS-1}\Odd(\cR) \prod_{w\in
\cS}(1-\Odd(w)).
$$
Denote
$$h^\cS(H)=M^\cS(H)+U^\cS(H)\frac{q}{q+1}.$$


\subsection{Component groups at $x$ and $y$}
Let $D$ and $R$ be as in $\S$\ref{ssDES}. Recall that we assume
$\infty\not\in R$. Fix a place $w\in R$. Let $D^w$ be the quaternion
algebra over $F$ which is ramified at $(R-w)\cup \infty$. Fix a
maximal $A$-order $\fD$ in $D^w$, and denote
\begin{align*}
&A^w=A[w^{-1}];\\
&\fD^w=\fD\otimes_A A^w;\\
&\G^w=\{\gamma\in (\fD^w)^\times\ |\ \ord_w(\Nr(\gamma))\in 2\Z\};
\end{align*}
here $w^{-1}$ denotes the inverse of a generator of the ideal in $A$
corresponding to $w$, and $\Nr$ denotes the reduced norm on $D^w$.

By fixing an isomorphism $D^w\otimes_F F_w\cong \M_2(F_w)$, one can
consider $\G^w$ as a subgroup of $\GL_2(F_w)$ whose image in $\PGL_2(F_w)$ 
is discrete and cocompact. Hence $\G^w$ acts on the Bruhat-Tits tree $\cT^w$ of $\PGL_2(F_w)$.
It is not hard to show that $\G^w$ acts without inversions, so the
quotient graph $\G^w\bs \cT^w$ is a finite graph without loops. We make $\G^w\bs
\cT^w$ into a graph with lengths by assigning to each edge $e$ of
$\G^w\bs \cT^w$ the length
$\#(\Stab_{\G^w}(\tilde{e})/\F_q^\times)$, where $\tilde{e}$ is a
preimage of $e$ in $\cT^w$. The graph with lengths $\G^w\bs \cT^w$
does not depend on the choice of isomorphism $D^w\otimes_F F_w\cong
\M_2(F_w)$, since such isomorphisms differ by conjugation.

As follows from the analogue of Cherednik-Drinfeld uniformization
for $X^R_{F_w}$, proven in this context by Hausberger
\cite{Hausberger}, $X^R_{F_w}$ is a twisted Mumford curve: Denote by
$\cO_w^{(2)}$ the quadratic unramified extension of $\cO_w$ and
denote by $\F_w^{(2)}$ the residue field of $\cO_w^{(2)}$. Then
$X^R_{F}$ has a semi-stable model $X^R_{\cO_w^{(2)}}$ over
$\cO_w^{(2)}$ such that the irreducible components of
$X^R_{\F_w^{(2)}}$ are projective lines without self-intersections,
and the dual graph $G(X^R_{\cO_w^{(2)}})$, as a graph with lengths,
is isomorphic to $\G^w\bs \cT^w$.

On the other hand, as is done in \cite{Kurihara} for the quaternion
algebras over $\Q$, the structure of $\G^w\bs \cT^w$ can be related
to the arithmetic to $D^w$: The number of vertices of $\G^w\bs
\cT^w$ is $2h^\emptyset(D^w)$, the number of edges is $h^w(D^w)$,
each edge has length $1$ or $q+1$, and the number of edges of length
$q+1$ is $U^w(D^w)$ (the notation here is as in $\S$\ref{ssCN}).
Hence, using the formulae in $\S$\ref{ssCN}, we get the following:

\begin{prop}\label{propC-Dmodel}
$X^R_{F}$ has a semi-stable model $X^R_{\cO_w^{(2)}}$ over
$\cO_w^{(2)}$ such that $X^R_{\F_w^{(2)}}$ is a union of projective
lines without self-intersections. The number of vertices of the dual
graph $G(X^R_{\cO_w^{(2)}})$ is
$$
\frac{2}{q^2-1}\prod_{v\in R-w}(q_v-1)+2^{\#
R-1}\Odd(R-w)\frac{q}{q+1};
$$
the number of edges is
$$
\frac{(q_w+1)}{q^2-1}\prod_{v\in R-w}(q_v-1)+2^{\#
 R-1}\Odd(R-w)(1-\Odd(w))\frac{q}{q+1}.
$$
The edges of $G(X^R_{\cO_w^{(2)}})$ have length $1$ or $q+1$. The
number of edges of length $q+1$ is
$$
2^{\# R-1}\Odd(R-w)(1-\Odd(w)).
$$

\end{prop}

This proposition has an interesting corollary:

\begin{cor}\label{corGenus}
Let $g(R)$ be the genus of $X^R_F$. Then
$$
g(R)=1+\frac{1}{q^2-1}\prod_{v\in R}(q_v-1)-\frac{q}{q+1}2^{\#
 R-1}\Odd(R).
$$
\end{cor}
\begin{proof}
Let $h_1$ be the dimension of the first simplicial homology group of
$G(X^R_{\cO_w^{(2)}})$ with $\Q$-coefficients. Let $V, E$ be the
number of vertices and edges of this graph, respectively. By Euler's
formula, $h_1=E-V+1$. Proposition \ref{propC-Dmodel} gives formulae
for $V$ and $E$ from which it is easy to see that $h_1$ is given by
the above expression. Since the irreducible components of
$X^R_{\F_w^{(2)}}$ are projective lines, it is not hard to show
that $h_1$ is the arithmetic genus of $X^R_{\F_w^{(2)}}$; cf.
\cite[p. 298]{Hartshorne}. On the other hand, $X^R_{\cO_w^{(2)}}$
is flat over $\cO_w^{(2)}$, so the genus $g(R)$ of its generic fiber
is equal to the arithmetic genus of the special fiber; see \cite[p.
263]{Hartshorne}. (Note that the special role of $w$ in the formulae
for $V$ and $E$ disappears in $g(R)$, as expected. This formula for
$g(R)$ was obtained in \cite{PapGenus} by a different argument.)
\end{proof}

\begin{thm}\label{thmCGxy'}Let $\Phi_v'$ denote the group of connected components of
$J^{xy}$ at $v\in |F|$.
\begin{itemize}
\item[(i)] $\Phi_x'\cong \Z/(q+1)\Z$;
\item[(ii)] $\Phi_y'\cong \Z/(q^2+1)(q+1)\Z$.
\end{itemize}
\end{thm}
\begin{proof}
In general, the information supplied by Proposition
\ref{propC-Dmodel} is not sufficient for determining the graph
$G(X^R_{\cO_w^{(2)}})$ uniquely. Nevertheless, in the case when
$R=\{x, y\}$ Proposition \ref{propC-Dmodel} does uniquely determine
$G(X^R_{\cO_w^{(2)}})$: $G(X^{xy}_{\cO_x^{(2)}})$ is a graph
without loops, which has $2$ vertices, $q+1$ edges, and all edges
have length $1$. Similarly, $G(X^{xy}_{\cO_y^{(2)}})$ is a graph
without loops, which has $2$ vertices, $q+1$ edges, two of the edges
have length $q+1$ and all others have length $1$. Hence, in both
cases, the dual graph is the graph with two vertices and $q+1$ edges
connecting them, cf. Figure \ref{Fig3}.

\begin{figure}
\begin{picture}(40,25)
\qbezier(5,13)(20,35)(35,13)\qbezier(5,13)(20,25)(35,13)
\qbezier(5,13)(20,20)(35,13) \qbezier(5,13)(20,-10)(35,13)
\put(5,13){\circle*{2}}\put(35,13){\circle*{2}}
\put(20,12){\circle*{.7}}\put(20,9){\circle*{.7}}\put(20,6){\circle*{.7}}
\end{picture}
\caption{}\label{Fig3}
\end{figure}

Now Theorem \ref{thmCG} can be used to conclude that the component
groups are as stated.
\end{proof}


\subsection{Component group at $\infty$}\label{ssInftyD} Here we again rely on the
existence of analytic uniformization. Let $\La$ be a maximal
$A$-order in $D$. Let
$$
\G^\infty:=\La^\times.
$$
Since $D$ splits at $\infty$, by fixing an isomorphism $D\otimes
\Fi\cong \M_2(\Fi)$, we get an embedding $\G^\infty\hookrightarrow
\GL_2(\Fi)$. The group $\G^\infty$ is a discrete, cocompact subgroup of
$\GL_2(\Fi)$, well-defined up to conjugation. Let $\cT^\infty$ be
the Bruhat-Tits tree of $\PGL_2(\Fi)$. The group $\G^\infty$ acts on
$\cT^\infty$ without inversions, so the quotient
$\G^\infty\bs\cT^\infty$ is a finite graph without loops which we
make into a graph with lengths by assigning to an edge $e$ of 
$\G^\infty\bs\cT^\infty$ the length
$\#(\Stab_{\G^\infty}(\tilde{e})/\F_q^\times)$, where $\tilde{e}$ is
a preimage of $e$ in $\cT^\infty$. By a theorem of Blum and Stuhler
\cite[Thm. 4.4.11]{BS},
$$
(X^R_{\Fi})^\an\cong \G^\infty\bs \Omega.
$$
From this one deduces that $X^R_F$ has a semi-stable model
$X^R_{\cO_\infty}$ over $\cO_\infty$ such that the dual graph of
$X^R_{\cO_\infty}$, as a graph with lengths, is isomorphic to
$\G^\infty\bs\cT^\infty$, cf. \cite{Kurihara}. The structure of
$\G^\infty\bs\cT^\infty$ can be related to the arithmetic of $D$;
see \cite{PapLocProp}.

\begin{prop}\label{propC-Dmodelinf}
$X^R_{F}$ has a semi-stable model $X^R_{\cO_\infty}$ over
$\cO_\infty$ such that the special fiber $X^R_{\F_\infty}$ is a
union of projective lines without self-intersections. 
The number of vertices of the dual graph 
$G(X^R_{\cO_\infty})$ is
$$
\frac{2}{q-1}(g(R)-1)+\frac{q}{q-1}2^{\# R -1}\Odd(R);
$$
the number of edges is
$$
\frac{q+1}{q-1}(g(R)-1)+\frac{q}{q-1}2^{\# R -1}\Odd(R).
$$
All edges have length $1$.
\end{prop}
\begin{proof}
See Proposition 5.2 and Theorem 5.5 in \cite{PapLocProp}.
\end{proof}

\begin{thm}\label{thmCGinfty'}
$\Phi_\infty'\cong \Z/(q+1)\Z$.
\end{thm}
\begin{proof}
Applying Proposition \ref{propC-Dmodelinf} in the case $R=\{x, y\}$,
one easily concludes that $X^{xy}_F$ has a semi-stable model over
$\cO_\infty$ whose dual graph looks like Figure \ref{Fig3}: it has
$2$ vertices, $q+1$ edges, and all edges have length $1$. The
structure of $\Phi_\infty'$ now follows from Theorem \ref{thmCG}.
\end{proof}


\section{Jacquet-Langlands isogeny}\label{SecJLisog}

Let $D$ and $R$ be as in $\S$\ref{ssDES}. Let $X:=X^R_F$,
$X':=X_0(R)_F$, $J:=J^R$, $J':=J_0(R)$. Fix a separable closure
$F^\sep$ of $F$ and let $G_F:=\Gal(F^\sep/F)$. Let $p$ be the
characteristic of $F$ and fix a prime $\ell\neq p$. Denote by
$V_\ell(J)$ the Tate vector space of $J$; this is a $\Q_\ell$-vector
space of dimension $2g(R)$ naturally equipped with a continuous
action of $G_F$. Let $V_\ell(J)^\ast$ be the linear dual of
$V_\ell(J)$.

\begin{thm}\label{thmExistence}
There is a surjective homomorphism $J'\to J$ defined over $F$.
\end{thm}
\begin{proof}
Let $\A=\prod_{v\in |F|}' F_v$ denote the adele ring of $F$ and let
$\A^\infty=\prod_{v\in |F|-\infty}' F_v$, so $\A=\A^\infty\times
F_\infty$. Fix a uniformizer $\pi_\infty$ at $\infty$. Let
$\cA(D^\times(F)\bs D^\times(\A)/\pi_\infty^\Z)$ be the space of
$\overline{\Q}_\ell$-valued locally constant functions on
$D^\times(\A)/\pi_\infty^\Z$ which are invariant under the action of
$D^\times(F)$ on the left. This space is equipped with the right
regular representation of $D^\times(\A)/\pi_\infty^\Z$. Since $D$ is
a division algebra, the coset space $D^\times(F)\bs
D^\times(\A)/\pi_\infty^\Z$ is compact and decomposes as a sum of
irreducible admissible representations $\Pi$ with finite multiplicities $m(\Pi)>0$, cf. \cite[$\S$13]{LRS}:
\begin{equation}\label{eqAR}
\cA_D:=\cA(D^\times(F)\bs D^\times(\A)/\pi_\infty^\Z)=\bigoplus_\Pi
m(\Pi)\cdot \Pi.
\end{equation}
Moreover, as follows from the Jacquet-Langlands correspondence and the multiplicity-one theorem 
for automorphic cuspidal representations of $\GL_2(\A)$, the multiplicities $m(\Pi)$ are all equal to $1$; see \cite[Thm. 10.10]{Gelbart}. 
The representations appearing in the sum (\ref{eqAR}) are
called \textit{automorphic}. Each automorphic representation $\Pi$
decomposes as a restricted tensor product $\Pi=\otimes_{v\in |F|}
\Pi_v$ of admissible irreducible representations of $D^\times(F_v)$.
We denote $\Pi^\infty=\otimes_{v\neq \infty}\Pi_v$, so
$\Pi=\Pi^\infty\otimes \Pi_\infty$. If $\Pi$ is finite dimensional,
then it is of the form $\Pi=\chi\circ \Nr$, where $\chi$ is a Hecke
character of $\A^\times$ and $\Nr$ is the reduced norm on
$D^\times$, cf. \cite[Lem. 14.8]{LRS}. If $\Pi$ is infinite
dimensional, then $\Pi_v$ is infinite dimensional for every
$v\not\in R$.

Let $\psi_v$ be a character of $F_v^\times$. Denote by $\Sp_v\otimes
\psi_v$ the unique irreducible quotient of the induced
representation
$$
\mathrm{Ind}_{B}^{\GL_2}
(|\cdot|_v^{-\frac{1}{2}}\psi_v\oplus|\cdot|_v^{\frac{1}{2}}\psi_v),
$$
where $B$ is the subgroup of upper-triangular matrices in $\GL_2$.
The representation $\Sp_v\otimes \psi_v$ is called the
\textit{special representation} of $\GL_2(F_v)$ twisted by $\psi_v$.
If $\psi_v=1$, then we simply write $\Sp_v$.

For $v\in R$, let $\cD_v$ be the maximal order in $D(F_v)$. Let
$$
\cK:=\prod_{v\in R}\cD_v^\times\times \prod_{v\in
|F|-R-\infty}\GL_2(\cO_v)\subset D^\times(\A^\infty).
$$
Taking the $\cK$-invariants in Theorems 14.9 and 14.12 in
\cite{LRS}, we get an isomorphism of $G_F$-modules
\begin{equation}\label{eq-ESCD}
V_\ell(J)^\ast\otimes_{\Q_\ell}\overline{\Q}_\ell=
H_{\text{\'et}}^1(X\otimes_F F^\sep,
\overline{\Q}_\ell)=\bigoplus_{\substack{\Pi\in \cA_D\\
\Pi_\infty\cong \Sp_\infty}}\left(\Pi^\infty\right)^\cK\otimes
\sigma(\Pi),
\end{equation}
where $\sigma(\Pi)$ is a $2$-dimensional irreducible representation
of $G_F$ over $\overline{\Q}_\ell$ with the following property: If
$\left(\Pi^\infty\right)^\cK\neq 0$, then for all $v\in
|F|-R-\infty$, $\sigma(\Pi)$ is unramified at $v$ and there is an
equality of $L$-functions
$$
L(s-\frac{1}{2}, \Pi_v)=L(s, \sigma(\Pi)_v);
$$
here $\sigma(\Pi)_v$ denotes the restriction of $\sigma(\Pi)$ to a
decomposition group at $v$. This uniquely determines $\sigma(\Pi)$
by the Chebotarev density theorem \cite[Ch. I, pp.
8-11]{SerreAbelian}. Next, we claim that the dimension of
$\left(\Pi^\infty\right)^\cK$ is at most one. Indeed, if $v\in
|F|-R-\infty$, then $\Pi_v^{\GL_2(\cO_v)}$ is at most 
one-dimensional by \cite[Thm. 4.6.2]{Bump}. On the other hand, note that
$\cD_v^\times$ is normal in $D^\times(F_v)$ and
$D^\times(F_v)/\cD_v^\times\cong \Z$ for $v\in R$. Hence
$\Pi_v^{\cD_v^\times}\neq 0$ implies $\Pi_v=\psi_v\circ \Nr$ for
some unramified character of $F_v^\times$ ($\psi_v$ is unramified
because the reduced norm maps $\cD_v^\times$ surjectively onto
$\cO_v^\times$).

Let $\cI_v$ be the Iwahori subgroup of $\GL_2(\cO_v)$, i.e., the
subgroup of matrices which maps to $B(\F_v)$ under the reduction map
$\GL_2(\cO_v)\to \GL_2(\F_v)$. Let
$$
\cI=\prod_{v\in R} \cI_v\times \prod_{v\in |F|-R-\infty}\GL_2(\cO_v)
\subset\GL_2(\A^\infty).
$$
Let $\cA_0:=\cA_0(\GL_2(F)\bs \GL_2(\A))$ be the space of
$\overline{\Q}_\ell$-valued cusp forms on $\GL_2(\A)$; see
\cite[$\S$4]{GR} or \cite[$\S$3.3]{Bump} for the definition. Taking
the $\cI$-invariants in Theorem 2 of \cite{Drinfeld}, we get an
isomorphism of $G_F$-modules
\begin{equation}\label{eq-modDR}
V_\ell(J')^\ast\otimes_{\Q_\ell}\overline{\Q}_\ell=H_{\text{\'et}}^1(X'\otimes_F
F^\sep,
\overline{\Q}_\ell)=\bigoplus_{\substack{\Pi\in \cA_0\\
\Pi_\infty\cong \Sp_\infty}}\left(\Pi^\infty\right)^{\cI}\otimes
\rho(\Pi),
\end{equation}
where $\rho(\Pi)$ is $2$-dimensional irreducible representation of
$G_F$ over $\overline{\Q}_\ell$ with the following property: If
$\left(\Pi^\infty\right)^\cI\neq 0$, then for all $v\in
|F|-R-\infty$, $\rho(\Pi)$ is unramified at $v$ and
$$
L(s-\frac{1}{2}, \Pi_v)=L(s, \rho(\Pi)_v).
$$
In this case, $\left(\Pi^\infty\right)^{\cI}$ is finite dimensional,
but its dimension might be larger than one (due to the existence of
old forms).

The global Jacquet-Langlands correspondence \cite[Ch. III]{JaLa}
associates to each infinite dimensional automorphic representation
$\Pi$ of $D^\times(\A)$ a cuspidal representation
$\Pi'=\mathrm{JL}(\Pi)$ of $\GL_2(\A)$ with the following
properties:
\begin{enumerate}
\item if $v\not\in R$ then $\Pi_v\cong\Pi'_v$;
\item if $v\in R$ and $\Pi_v\cong \psi_v\circ\Nr$ for a character
$\psi$ of $F_v^\times$, then $$\Pi_v'\cong \Sp_v\otimes \psi_v.$$
\end{enumerate}

As we observed above, for $\Pi\in \cA_D$ such that
$(\Pi^\infty)^\cK\neq 0$, the characters $\psi_v$ at the places in
$R$ are unramified. Thus, for $v\in R$, $\Pi_v'$ is a twist of
$\Sp_v$ by an unramified character. On the other hand, the
representations of the form $\Sp_v\otimes \psi_v$, with $\psi_v$
unramified, can be characterized by the property that they have a
unique $1$-dimensional $\cI_v$-fixed subspace; see \cite{Casselman}.
Hence if $(\Pi^\infty)^\cK\neq 0$, then $((\Pi')^\infty)^\cI\neq 0$.

Now using (\ref{eq-ESCD}) and (\ref{eq-modDR}), one concludes that
$V_\ell(J)$ is isomorphic with a quotient of $V_\ell(J')$ as a
$G_F$-module. On the other hand, by a theorem of Zarhin (for $p>2$)
and Mori (for $p=2$)
\begin{equation}\label{eq-isog}
\Hom_F(J', J)\otimes \Q_\ell\cong \Hom_{G_F}(V_\ell(J'), V_\ell(J)).
\end{equation}
Thus, there is a surjective homomorphism $J'\to J$ defined over $F$.
\end{proof}

\begin{cor}
$J_0(xy)$ and $J^{xy}$ are isogenous over $F$.
\end{cor}
\begin{proof} Since $\dim(J^{xy})=q=\dim(J_0(xy))$, the claim follows from Theorem
\ref{thmExistence}.
\end{proof}

\begin{conj}\label{myConj} There exists an isogeny $J_0(xy)\to J^{xy}$ whose
kernel is $\cC_0$.
\end{conj}

As an initial evidence for the conjecture, note that $J_0(xy)/\cC_0$
has component groups at $x, y, \infty$ of the same order as those of
$J^{xy}$. This follows from Theorem \ref{propGroth}, Corollary
\ref{lemC0}, and Table \ref{table0} in the introduction. We will show below
that Conjecture \ref{myConj} is true for $q=2$.

\begin{rem}\label{rem-new} The statement of Theorem \ref{thmExistence} can be refined.
The abelian variety $J$ has toric reduction at every $v\in R$, so it is isogenous to an
abelian subvariety of $J'$ having the same reduction property. The
new subvariety of $J'$, $J'^\new$, defined as in the case of
classical modular Jacobians (cf. \cite{Ribet90}, \cite[p.
248]{GekelerJL}), is the abelian subvariety of $J'$ of maximal
dimension having toric reduction at every $v\in R$. Hence $J$ is
isogenous to a subvariety of $J'^\new$. By computing the dimension
of $J'^\new$, one concludes that $J$ and $J'^\new$ are isogenous
over $F$.
\end{rem}

\begin{rem}
There is just one other case, besides the case which is the focus of this paper, 
when $J$ and $J'$ are actually isogenous. As one easily shows by comparing 
the genera of modular curves $X^R$ and $X_0(R)$, the genera 
of these curves are equal if and only if $R=\{x, y\}$ and 
$\{\deg(x),\deg(y)\}=\{1,1\}, \{1,2\}, \{2,2\}$. Assume $\deg(x)=\deg(y)=2$. 
Then the genus of both $X^{xy}$ and $X_0(xy)$ is $q^2$, 
but neither of these curves is hyperelliptic. The curve $X_0(xy)$ again has $4$ cusps 
which can be represented as in $\S$\ref{SecCDG}. 
Calculations similar to those we have  
carried out in earlier sections lead to the following result: 
\begin{enumerate}
\item The cuspidal divisor group $\cC$ is generated by $c_0$ and $c_x$. 
The order of  $c_0$ is $q^2+1$. 
The order of $c_x$ is divisible by $q^2+1$ and divides $q^4-1$. 
The order of $c_y$ is divisible by $q^2+1$ and divides $q^4-1$. 
\item $\Phi_x\cong \Phi_x'\cong \Z/(q^2+1)\Z$. 
\item $\Phi_y\cong \Phi_y'\cong \Z/(q^2+1)\Z$.
\item The canonical map $\phi_x:\cC\to \Phi_x$ is surjective, and 
$$
\phi_x(c_0)=z, \quad \phi_x(c_x)=0, \quad \phi_x(c_y)=z.
$$ 
\item The canonical map $\phi_y:\cC\to \Phi_y$ is surjective, and 
$$
\phi_x(c_0)=z, \quad \phi_y(c_x)=z, \quad \phi_y(c_y)=0.
$$
\end{enumerate}
The fact that $X_0(xy)$ is not hyperelliptic complicates the calculation of $\cC$: 
just the relations between the cuspidal divisors arising from the Drinfeld discriminant function 
are not sufficient for pinning down the orders of $c_x$ and $c_y$, cf. (\ref{eq-rel1}). Next, 
the calculations required for determining $\Phi_\infty$, $\Phi_\infty'$, and $\phi_\infty$ 
appear to be much more complicated than those in 
$\S$\ref{ssInfty} and $\S$\ref{ssInftyD}.  Nevertheless, based on the facts that we are able to prove, 
and in analogy with the case $\deg(x)=1$, $\deg(y)=2$, 
we make the following prediction: The cuspidal divisor 
group $\cC\cong (\Z/(q^2+1)\Z)^2$ 
is the direct sum of the cyclic subgroups generated by $c_x$ and $c_y$ 
both of which have order $q^2+1$, 
and there is an isogeny $J_0(xy)\to J^{xy}$ whose kernel is $\cC$. 
\end{rem}

\begin{defn}
It is known that every elliptic curve $E$ over $F$ with conductor
$\fn_E=\fn\cdot \infty$, $\fn\lhd A$, and split multiplicative
reduction at $\infty$ is isogenous to a subvariety of $J_0(\fn)$;
see \cite{GR}. This follows from (\ref{eq-modDR}), (\ref{eq-isog}),
and the fact \cite[p. 577]{Deligne} that the representation $\rho_E:
G_F\to \Aut(V_\ell(E)^\ast)$ is automorphic (i.e.,
$\rho_E=\rho(\Pi)$ for some $\Pi\in \cA_0$). The multiplicity-one
theorem can be used to show that in the $F$-isogeny class of $E$
there exists a unique curve $E'$ which is isomorphic to a
one-dimensional abelian subvariety of $J_0(\fn)$, thus maps
``optimally'' into $J_0(\fn)$. We call $E'$ the
\textit{$J_0(\fn)$-optimal} curve. Theorem \ref{thmExistence} and
Remark \ref{rem-new} imply that $E$ with square-free conductor
$R\cdot \infty$ and split multiplicative reduction at $\infty$ is
also isogenous to a subvariety of $J^R$. Moreover, in the
$F$-isogeny class of $E$ there is a unique elliptic curve $E''$
which is isomorphic to a one-dimensional abelian subvariety of
$J^R$. We call $E''$ the \textit{$J^R$-optimal} curve.
\end{defn}

\begin{notn}
Let $E$ be an elliptic curve over $F$ given by a Weierstrass
equation
$$
E:\ Y^2+a_1XY+a_3Y=X^3+a_2X^2+a_4X+a_6.
$$
Let $E^{(p)}$ be the elliptic curve given by the equation
$$
E^{(p)}:\ Y^2+a_1^pXY+a_3^pY=X^3+a_2^pX^2+a_4^pX+a_6^p.
$$
There is a Frobenius morphism $\Frob_p: E\to E^{(p)}$ which maps a
point $(x_0, y_0)$ on $E$ to the point $(x_0^p, y_0^p)$ on
$E^{(p)}$. It is clear that the $j$-invariants of these elliptic
curves are related by the equation $j(E^{(p)})=j(E)^p$. If $E$ has
semi-stable reduction at $v\in |F|$, then $\Phi_{E, v}\cong \Z/n\Z$,
where $\Phi_{E,v}$ denotes the component group of $E$ at $v$ and
$n=-\ord_v(j(E))\geq 1$. In this case, $\Phi_{E^{(p)}, v}\cong
\Z/pn\Z$.
\end{notn}

\begin{defn}
An elliptic curve $E$ over $F$ with $j$-invariant $j(E)\not\in \F_q$
is said to be \textit{Frobenius minimal} if it is not isomorphic to
$\tE^{(p)}$ for some other elliptic curve $\tE$ over $F$. It is easy
to check that this is equivalent to $j(E)\not\in F^p$, cf.
\cite{SchweizerStrong}.
\end{defn}

For $q$ even, Schweizer has completely classified the elliptic
curves over $F$ having conductor of degree $4$ in terms of explicit
Weierstrass equations; see \cite{Schweizer2}. We are particularly
interested in those curves which have conductor $xy\infty$ and split
multiplicative reduction at $\infty$.

\begin{thm}\label{thmECurves2} Assume $q=2^s$.
Elliptic curves over $F$ with conductor $xy\infty$ exist only if
there exists an $\F_q$-automorphism of $F$ that transforms the
conductor into $(T+1)(T^2+T+1)\infty$. In particular, $s$ must be
odd.

If $s$ is odd, then there exists two isogeny classes of elliptic
curves over $F$ with conductor $(T+1)(T^2+T+1)\infty$ and split
multiplicative reduction at $\infty$. The Frobenius minimal curves
in each isogeny class are listed in Tables \ref{table1} and
\ref{table2}; the last three columns in the tables give the orders
of the component groups $\Phi_{E,v}$ of the corresponding curve $E$
at $v=x, y, \infty$.
\end{thm}
\begin{proof}
Theorem 4.1 in \cite{Schweizer2}.
\end{proof}

\begin{table}
\begin{tabular}{|c|c|c|c|c|}
\hline
 & Equation & $x$ & $y$ & $\infty$\\
\hline
$E_1$ & $Y^2+TXY+Y=X^3+T^3+1$ & 3 & 3 & 3 \\
$E_1'$ & $Y^2+TXY+Y=X^3+T^2(T^3+1)$ & 9 & 1 & 1\\
$E_1''$ & $Y^2+TXY+Y=X^3$ & 1 & 1 & 9\\
\hline
\end{tabular}\caption{Isogeny class I}\label{table1}
\end{table}

\begin{table}
\begin{tabular}{|c|c|c|c|c|}
\hline
 & Equation & $x$ & $y$ & $\infty$\\
\hline
$E_2$ & $Y^2+TXY+Y=X^3+X^2+T$ & 5 & 1 & 5\\
$E_2'$ & $Y^2+TXY+Y=X^3+X^2+T^5+T^2+T$ & 1 & 5 & 1\\
\hline
\end{tabular}\caption{Isogeny class II}\label{table2}
\end{table}

Next, \cite[Prop. 3.5]{Schweizer2} describes explicitly the
isogenies between the curves in classes I and II: There is an
isomorphism of \'etale group-schemes over $F$
$$
E_1[3]\cong H_1\oplus H_2,
$$
where $H_1\cong \Z/3\Z$ and $H_2\cong \mu_3$. The subgroup-scheme $H_1$ is generated by
$(T+1, 1)$ and $H_2$ is generated by $(T^2, sT^3+s^2)$, where $s$ is
a third root of unity. Then $E_1/H_1\cong E_1'$ and $E_1/H_2\cong
E_1''$. (It is well-known that an elliptic curve over $F$ with
conductor of degree $4$ has rank $0$, so in fact $E_1(F)=H_1\cong
\Z/3\Z$.) Similarly, the subgroup-scheme $H_3$ of $E_2$ generated by
$(1,1)$ is isomorphic to $\Z/5\Z$, $E_2/H_3\cong E_2'$, and
$E_2(F)=H_3\cong \Z/5\Z$.

\begin{prop}\label{propOpt}
Assume $q=2^s$ and $s$ is odd.
\begin{enumerate}
\item[(i)] $E_1$ and $E_2$ are the $J_0(xy)$-optimal curves in the
isogeny classes I and II.
\item[(ii)] $E_2'$ is the $J^{xy}$-optimal curve in the isogeny class
II.
\item[(iii)] If Conjecture \ref{myConj} is true, then $E_1$ is the
$J^{xy}$-optimal curve in the isogeny class I.
\end{enumerate}
\end{prop}
\begin{proof}
(i) There is a method due to Gekeler and Reversat \cite[Cor.
3.19]{Analytic} which can be used to determine $\# \Phi_{E, \infty}$
of the $J_0(\fn)$-optimal curve in a given isogeny class. This
method is based on the study of the action of Hecke algebra on
$H_1(\G_0(\fn)\bs \cT, \Z)$. For $\deg(\fn)=3$ the Gekeler-Reversat
method can be further refined \cite[Cor. 1.2]{SchweizerStrong}.
Applying this method for $\fn=xy$, one obtains $\#\Phi_{E,\infty}=3$
(resp. $\#\Phi_{E,\infty}=5$) for the $J_0(xy)$-optimal elliptic
curve $E$ in the isogeny class I (resp. II). Since there is a unique
curve with this property in each isogeny class, we conclude that
$E_1$ and $E_2$ are the $J_0(xy)$-optimal elliptic curves. (For
$q=2$, this is already contained in \cite[Example 4.4]{Analytic}.)

(ii) Assume $q$ is arbitrary. Let $E$ be an elliptic curve over $F$
which embeds into $J^{xy}$. Since $J^{xy}$ has split toric reduction
at $\infty$, \cite[Cor. 2.4]{PapTAMS} implies that the kernel of the
natural homomorphism
$$\Phi_{E, \infty}\to \Phi'_\infty\cong \Z/(q+1)\Z$$ is a subgroup of
$\Z/(q_\infty-1)\Z$. Hence $\#\Phi_{E, \infty}$ divides $(q^2-1)$.
First, this implies that $\#\Phi_{E, \infty}$ is coprime to $p$, so
$E$ must be Frobenius minimal in its isogeny class. Second, if
$q=2^s$ and $s$ is odd, then $5$ does not divide $(q^2-1)$, so $E_2$
is not $J^{xy}$-optimal. This leaves $E_2'$ as the only possible
$J^{xy}$-optimal curve in the isogeny class II.

(iii) Let $E$ be the $J^{xy}$-optimal curve in the isogeny class I.
By the discussion in (ii), this curve is one of the curves in Table
\ref{table1}. Suppose there is an isogeny $\varphi: J_0(xy)\to
J^{xy}$ whose kernel is $\cC_0$. Restricting $\varphi$ to
$E_1\hookrightarrow J_0(xy)$, we get an isogeny $\varphi': E_1\to E$
defined over $F$ whose kernel is a subgroup of $\cC_0\cong
\Z/(q^2+1)\Z$. Note that $3$ does not divide $q^2+1$. On the other
hand, any isogeny from $E_1$ to $E_1'$ or $E_1''$ must have kernel
whose order is divisible by $3$. This implies that $\varphi'$ has
trivial kernel, so $E=E_1$.
\end{proof}

\begin{rem}
In the notation of the proof of Proposition \ref{propOpt}, consider
the restriction of $\varphi$ to $E_2\hookrightarrow J_0(xy)$. By
part (ii) of the proposition, there results an isogeny $\varphi'':
E_2\to E_2'$ whose kernel is a subgroup of $\Z/(q^2+1)\Z$. Since $5$
divides $q^2+1$ when $s$ is odd, part (ii) of Proposition
\ref{propOpt} is compatible with Conjecture \ref{myConj}.
\end{rem}

\begin{thm}\label{propq=2}
Conjecture \ref{myConj} is true for $q=2$.
\end{thm}
\begin{proof}
Assume $q=2$. By Proposition \ref{propOpt}, $E_1$ and $E_2$ are the
$J_0(xy)$-optimal curves. Since the genus of $X_0(xy)$ is $2$, it is
hyperelliptic (this is true for general $q$ by Schweizer's theorem
which we used in $\S$\ref{SecCDG}). The genus being $2$ also implies
that a quotient of $X_0(xy)$ by an involution has genus $0$ or $1$.
The Atkin-Lehner involutions form a subgroup in $\Aut(X_0(xy))$
isomorphic to $(\Z/2\Z)^2$. Since the hyperelliptic involution is
unique, each $E_1$ and $E_2$ can be obtained as a quotient of
$X_0(xy)$ under the action of an Atkin-Lehner involution. Thus,
there are degree-$2$ morphisms $\pi_i:X_0(xy)\to E_i$, $i=1,2$. In
fact, one obtains the closed immersions $\pi_i^\ast:E_i\to J_0(xy)$
from these morphisms by Picard functoriality. Let
$\widehat{\pi_i^\ast}: J_0(xy)\to E_i$ be the dual morphism. It is
easy to show that the composition
$\widehat{\pi_i^\ast}\circ\pi_i^\ast:E_i\to E_i$ is the isogeny
given by multiplication by $2=\deg(\pi_i)$. This implies that $E_1$
and $E_2$ intersect in $J_0(xy)$ in their common subgroup-scheme of
$2$-division points $S:=\pi_1^\ast(E_1)[2]=\pi_2^\ast(E_2)[2]$, so
$$
J_0(xy)(F)=H_1\oplus H_3\cong \Z/3\Z\oplus \Z/5\Z=\cC.
$$
Let $\psi: J_0(xy)\to E_1\times E_2$ be the isogeny with kernel $S$.
Note that $S$ is characterized by the non-split exact sequence of
group-schemes over $F$:
$$
0\to\mu_2\to S\to \Z/2\Z\to 0.
$$

By Proposition \ref{propOpt}, $E_2'$ is the $J^{xy}$-optimal
elliptic curve in the isogeny class II. Let $E$ be the
$J^{xy}$-optimal elliptic curves in class I. From the proof of
Proposition \ref{propOpt}, we know that $E$ is Frobenius minimal, so
it is one of the curves listed in Table \ref{table1}. There are also
Atkin-Lehner involutions acting on $X^{xy}$ and they form a subgroup
in $\Aut(X^{xy})$ isomorphic to $(\Z/2\Z)^2$; see \cite{PapHE}. Now
exactly the same argument as above implies that $E$ and $E_2'$
intersect in $J^{xy}$ along their common subgroup-scheme of
$2$-division points $S'\cong S$. Let $\nu: J^{xy}\to E\times E_2'$
be the isogeny with kernel $S'$. Let $\hat{\nu}: E\times E_2' \to
J^{xy}$ be the dual isogeny.

The following argument is motivated by \cite{GoRo}. Consider the
composition
$$
\phi:J_0(xy)\xrightarrow{\psi} E_1\times E_2
\xrightarrow{\phi_1\times\phi_2} E\times E_2'
\xrightarrow{\hat{\nu}} J^{xy},
$$
where $\phi_1$ is either the identity morphism or has kernel $H_1$,
$H_2$, and $\phi_2$ has kernel $H_3$. Since $\phi_1\times\phi_2$ has
odd degree, this morphism maps the kernel of $\hat{\psi}$ to the
kernel of $\hat{\nu}$. Indeed, both are the ``diagonal'' subgroups
isomorphic to $S$ in the corresponding group-schemes $(E_1\times
E_2)[2]$ and $(E\times E_2')[2]$. More precisely,
$\cH:=\ker(\hat{\psi})$ is uniquely characterized as the
subgroup-scheme of $\cG:=(E_1\times E_2)[2]$ having the following
properties: $\cH^0$ is the image of the diagonal morphism $\mu_2\to
\mu_2\times \mu_2= \cG^0$ and the image of $\cH$ in $\cG^\et$ under
the natural morphism $\cG\to \cG^\et$ is the image of the diagonal
morphism $\Z/2\Z\to \Z/2\Z\times \Z/2\Z$. A similar description
applies to $\ker(\hat{\nu})\subset (E\times E_2')[2]$. Thus, there
is an isogeny $\phi':J_0(xy)\to J^{xy}$ such that $\phi=\phi'[2]$
and $\ker(\phi')\cong \ker(\phi_1\times\phi_2)$. We conclude that
$J^{xy}$ is isomorphic to the quotient of $J_0(xy)$ by one of the
following subgroups
$$
H_3, \quad H_1\oplus H_3,\quad H_2\oplus H_3.
$$
Now note that $H_1$ and $H_3$ under the specialization map
$\phi_\infty$ inject into $\Phi_\infty$, but $H_2$ maps to $0$
(indeed, $H_2\cong \mu_3$ has non-trivial action by
$\Gal(\bar{\F}_\infty/\F_\infty)$ whereas $\Phi_\infty$ is
constant). Hence Theorem \ref{propGroth} implies that the quotients
of $J_0(xy)$ by the subgroups listed above have component groups at
$\infty$ of orders $ 3, 1, 9$, respectively. Since
$\Phi_\infty'\cong \Z/3\Z$, we see that $J^{xy}$ is the quotient of
$J_0(xy)$ by $H_3$ which is $\cC_0$ in this case.
\end{proof}



\end{document}